\documentclass[envcountsect,envcountsame,oribibl,orivec]{llncs}
\usepackage{amssymb}
\usepackage{amsmath}
\usepackage{url}

\usepackage[right=2.9cm,left=2.9cm]{geometry}

\usepackage{bussproofs}

\EnableBpAbbreviations
\usepackage{pdfsync}
\usepackage{graphics}
 \usepackage{graphicx}
\usepackage{hyperref}
\usepackage{etoolbox}
\makeatletter
\let\llncs@addcontentsline\addcontentsline
\patchcmd{\maketitle}{\addcontentsline}{\llncs@addcontentsline}{}{}
\patchcmd{\maketitle}{\addcontentsline}{\llncs@addcontentsline}{}{}
\patchcmd{\maketitle}{\addcontentsline}{\llncs@addcontentsline}{}{}
\makeatother
\usepackage{hyperref}
\usepackage{bookmark}
  \usepackage{qtree}
\usepackage{tikz}

\makeatletter
\newcommand{\myLines}[1]{% Three-way only
\begin{picture}(4,1)
\put(0,0){\line(2,1){2}}
\put(2,0){\vector(0,1){0.7}}
\put(4,0){\line(-2,1){2}}
\end{picture}}
\let\qdrawReal=\qdraw@branches
\newcommand\brOverride{\let\qdraw@branches=\myLines}
\newcommand\brRestore{\let\qdraw@branches=\qdrawReal}
\makeatother
\makeatletter
\newcommand{\todo}[1]{\marginpar{\textbf{TODO\footnotemark}}\@latex@warning{TODO: #1}\footnotetext{ #1}}
\makeatother
\newcommand{\Prop}{\textsf{Prop}}
\newcommand{\Formulae}{\textsf{Fm}}
\newcommand{\Terms}{\textsf{Tm}}
\newcommand{\CTerms}{\textsf{Const}}
\newcommand{\VTerms}{\textsf{Var}}
\renewcommand{\phi}{\varphi}

\newcommand{\TermsBool}{\textsf{Term}}
\newcommand{\FormulaeBool}{\textsf{Form}}

\newcommand{\JL}{{\sf JL}}
\newcommand{\LP}{{\sf LP}}
\newcommand{\HLP}{{\sf HLP}}
\newcommand{\LPBool}{{\sf LP}^{B}}

\newcommand{\M}{{\mathcal M}}
\newcommand{\E}{{\mathcal E}}
\newcommand{\V}{{\mathcal V}}
\newcommand{\VExt}{\tilde{\V}}

\newcommand{\Val}{\theta}
\newcommand{\ValExt}{\tilde{\theta}}
\newcommand{\CS}{{\sf CS}}

\newcommand{\BIValExt}{\tilde{\theta}_{v}}
\newcommand{\BIExt}{\tilde{\theta}}
\newcommand{\T}{{\mathcal T}}

\newcommand{\limplies}{\rightarrow}
\newcommand{\liff}{\leftrightarrow}
\newcommand{\lequal}{\approx}

\newcommand{\tapp}{\cdot}
\newcommand{\ttsum}{+}
\newcommand{\tinspect}{!}
\newcommand{\jbox}[1]{#1:\!} 

\newcommand{\Ameet}{\otimes}
\newcommand{\Ajoin}{\oplus}
\newcommand{\Anot}{\ominus}
\newcommand{\Aimplise}{\Rightarrow}
\newcommand{\AbotElement}{0}
\newcommand{\AtopElement}{1}

\newcommand{\Abot}{\oldstylenums{0}}
\newcommand{\Atop}{\oldstylenums{1}}

\newcommand{\justVarOne}{\mathsf{x}}
\newcommand{\justVarTwo}{\mathsf{y}}
\newcommand{\justVarThree}{\mathsf{z}}

\newcommand{\justConsOne}{\mathsf{a}}
\newcommand{\justConsTwo}{\mathsf{b}}
\newcommand{\justConsThree}{\mathsf{c}}

\newcommand{\AElementOne}{a}
\newcommand{\AElementTwo}{b}

\newcommand{\Algebra}{{\mathcal A}}

\newcommand{\OBox}[1]{\Box_{#1}}

\newcommand{\OVal}[1]{\theta({#1})}
\newcommand{\OValExt}[1]{\tilde{\theta}({#1})}
\newcommand{\OCl}[1]{[{#1}]}

\newcommand{\Logic}[1]{\mathsf{#1}} 
\newcommand{\JLCS}[2]{{\sf #1}_{#2}}
\newcommand{\powerset}{\mathcal{P}}

\begin{document}
\title{Algebraic Semantics for the Logic of Proofs}

\author{Amir Farahmand Parsa\inst{1}\thanks{This research was in part supported by a grant from IPM and carried out in IPM-Isfahan Branch.} and Meghdad Ghari\inst{1,2}\thanks{This research was in part supported by a grant from IPM (No.99030420) and carried out in IPM-Isfahan Branch.}}
\institute{School of Mathematics, Institute for Research in Fundamental Sciences (IPM),  Tehran, Iran   \and    Department of Philosophy, Faculty of Literature and Humanities,
	University of Isfahan, Isfahan, Iran\\ \email{a.parsa@ipm.ir, ghari@ipm.ir}
}

\maketitle
\begin{abstract}
We present algebraic semantics for the classical logic of proofs based on Boolean algebras. We also extend the language of the logic of proofs in order to have a Boolean structure on proof terms and equality predicate on terms. Moreover, the completeness theorem and certain generalizations of Stone's representation theorem are obtained for all proposed algebras.
\end{abstract}

\section{Introduction}

Justification logics are modal-like logics that provide a framework for reasoning about epistemic justifications (see \cite{ArtFit11SEP,Art-Fit-Book-2019,Kuz-Stu-Book-2019}). The language of justification logics extends the language of propositional logic by proof terms and expressions of the form $t:A$, with the intended meaning ``$t$ is a justification for $A$'' or  ``$t$ is a proof for $A$''. The \emph{logic of proofs} $\LP$ is the first introduced logic in the family of justification logics due to  Artemov (\cite{Art95TR,Art01BSL}). The logic of proofs is a justification counterpart of modal logic $\Logic{S4}$. Other modal logics have  justification counterparts too (cf. \cite{Bre00TR,BrueGoeKuz10AiML,Fit16APAL,Ghari-APAL-2017,GoeKuz12APAL,Pac05PLS,Rub06CSR}).  

Various semantics have been proposed for the logic of proofs: arithmetical semantics (\cite{Art01BSL}), Mkrtychev models (\cite{Mkr97LFCS}), Fitting possible world models (\cite{Fit05APAL}), modular models (\cite{Art12SLnonote,KuzStu12AiML}), subset models (\cite{LehmannStuder2019}), game semantics (\cite{Ren05TR}), etc. In this paper, we aim to propose an algebraic semantics for the logic of proofs. The only known algebraic semantics for justification logics is due to Baur and Studer \cite{Baur-Studer-2020,Baur-Studer-JLC-2021} and Pischke \cite{Pischke-arXiv-2020}. 

Baur and Studer (\cite{Baur-Studer-2020,Baur-Studer-JLC-2021}) impose a semiring structure on evidence by adding axioms of semirings (on proof terms) to a basic justification logic. The resulted logic, called ${\sf SE}$, is proved to be sound and complete with respect to semiring models (a semiring equipped with an interpretation function, an evidence relation, and a truth assignment). In the justification logic framework, it is common to relativize logics with a set of justified axioms called constant specification. There are various kinds of constant specifications, some of them are used to show the \textit{internalization property}: every theorem is justified by a proof term.\footnote{In fact, the structure of this term reflects exactly the structure of the axiomatic proof of the theorem.} In the logic ${\sf SE}$ no constant specification  is mentioned in its standard form. However, the role of constant specifications is transfered to the set of assumptions.

Pischke (\cite{Pischke-arXiv-2020}) introduces algebraic variants of Mkrtychev models, Fitting models, and subset models for various intermediate justification logics. Since the main focus of \cite{Pischke-arXiv-2020} is on intuitionistic and intermediate justification logics, these algebraic variants are defined over Heyting algebras.

In this paper, we present algebraic semantics for the logic of proofs based on Boolean algebra. In modal logic, a modal algebra $A$ is a Boolean algebra equipped with the operator $\Box : A \to A$  that satisfies certain conditions. The standard method of proving completeness of modal logics with respect to modal algebras is to construct an algebra, called the Tarski-Lindenbaum algebra, out of formulas of the logic. The proof of well-definedness of the operator $\Box$ in the Tarski-Lindenbaum algebra follows from the fact that  the regularity rule is admissible in all normal modal logics:
\[\frac{\phi \leftrightarrow \psi}{\Box \phi \leftrightarrow \Box \psi}\ Reg.\]
Since the justification counterpart of this rule, namely
\[\frac{\phi \leftrightarrow \psi}{t:\phi \leftrightarrow t:\psi}\ JReg,\]
is not admissible in justification logics, we use operators on formulas instead of operators on the carrier of algebras. In Section \ref{sec: LP-CS algebras}, we present algebras for the logic of proofs, called full $\LP$ algebras, and we prove the completeness theorem. We also establish a generalization of Stone's representation theorem and show that the logic of proofs is complete with respect to set algebras. The full $\LP$ algebras are similar to algebraic Mkrtychev models of Pischke in \cite{Pischke-arXiv-2020}.

In Section \ref{sec: LP-CS algebras over arbitrary Boolean algebras}, we extend the language of $\LP$ to have a Boolean structure on proof terms. We also add the equality predicate (on proof terms) to the language. We then add axioms of Boolean algebra for terms to axioms of $\LP$. The resulting logic is denoted by $\LPBool$. Algebras for $\LPBool$, called full $\LPBool$ algebras, contain two extensions of Boolean algebras: one for proof terms and the other for formulas. Then completeness theorem and a generalization of Stone's representation theorem are proved. The results of Section \ref{sec: LP-CS algebras over arbitrary Boolean algebras} on full $\LPBool$ algebras are comparable to the work of Baur and Studer in (\cite{Baur-Studer-2020,Baur-Studer-JLC-2021}) on semiring modals, although $\LPBool$ algebras are defined slightly different from semiring models and further a Boolean algebra on proof terms are used instead of a semiring structure. 

Finally, in Section \ref{sec: LP-CS algebras over arbitrary polynomial Boolean algebras}, we consider the polynomial structure of proof terms and present an alternative class of algebras for $\LPBool$, called polynomial algebras. Again the completeness theorem and a representation theorem with respect to polynomial algebras are proved.

%%%%%%%%%%%%%%%%%%%%%%%%%%%%%%%%%%%%%%%%%%%%%%%%%%%%%%%
\section{Justification logics}\label{sec:Justification logics}

The language of justification logics is an
extension of the language of propositional logic by the formulas
of the form $t:F$, where $F$ is a formula and $t$ is a term. \textit{Proof terms} or \textit{justification terms} (or
\textit{terms} for short) are built up from (proof)
variables $\justVarOne, \justVarTwo, \justVarThree \ldots$ and (proof) constants $\justConsOne,\justConsTwo,\justConsThree,\ldots$  using several operators depending on the logic: (binary) application `$\tapp$', (binary) sum `$+$', and (unary) verifier `$!$'. %, (unary) negative verifier `$?$', and (unary) weak negative verifier `$\bar{?}$'. 
Proof terms are called \textit{proof polynomials} in \cite{Art95TR,Art01BSL}.

The binary operator $+$ \textit{combines} two justifications: $s+ t$ is a justification for everything justified by $s$ or by $t$. The binary operator $\tapp$ is used to internalize \textit{modus ponens}: if $s$ is a justification for $A \limplies B$ and $t$ is a justification for $A$, then $s\tapp t$ is a justification for $B$. The unary operator $!$ is a \textit{verifier}: if $t$ is a justification for $A$, then this fact can be verified by the justification $! t$.
 %The unary operators $!$ and $?$ are positive and negative \textit{verifiers} respectively: if $t$ is a justification for $A$, then this fact can be verified by the justification $! t$; if $t$ is not a justification for $A$, then this fact can be verified by the justification $? t$. The unary operator $\bar{?}$ is a kind of \textit{negative uniform verifier} for any false assertion: if $A$ is false, then, for any justification $t$,  the fact that $t$ is not a justification for $A$ can be verified by the justification $\bar{?} t$.

%Subterms of a term are defined in the usual way: $s$ is a subterm of $s,
%s+t, t+s, s\tapp t$, $!s$, $\bar{?}s$, and  $?s$. Moreover, we extend this subterm relation by transitivity. Let $!^0 t:= t$ and $!^n t := !!^{n-1} t$, for positive integer $n$.

Proof term and formulas are constructed by the following mutual grammar:  
\begin{gather*}
s, t  ::= \justConsThree \mid \justVarOne \mid s + t \mid  s \tapp t \mid   \ \tinspect t  \, , \\
\phi, \psi  ::= p \mid \bot \mid \neg \phi \mid \phi \vee \psi  \mid \jbox{t} \phi \, .
\end{gather*}
where $\justConsThree \in \CTerms$, $\justVarOne \in \VTerms$, and $p \in \Prop$. Other connectives $\wedge, \limplies, \liff$ are defined as usual. In particular, $\phi \limplies \psi := \neg \phi \vee \psi$. %This is the basic language of temporal justification logic that we are considering in this paper.

Let $\Terms$ and $\Formulae$ denote the set of all terms and the set of all formulas of $\LP$ respectively.

We now begin with describing the axiom schemes and rules of 
justification logics. The set of axiom schemes of the logic of proofs $\LP_\emptyset$ is:%\footnote{Axioms PL1-PL5 and rule MP constitutes the Russell--Bernays axiom system for propositional logic.} 

\begin{description}
%	\item[PL.] Finite set of axioms for classical propositional logic,
%	\item[PL1.] $\phi \limplies (\psi \limplies \phi)$,
%	\item[PL2.] $(\phi \limplies (\psi \limplies \chi) ) \limplies ( (\phi \limplies \psi) \limplies (\phi \limplies \chi) )$,
%	\item[PL3.] $\neg \neg \phi \limplies \phi$,
%	\item[PL3.] $\bot \limplies \phi$,
	\item[PL1.] $\phi \limplies (\phi \vee \psi)$,
	\item[PL2.] $(\phi \vee \psi) \limplies (\psi \vee \phi)$,
	\item[PL3.] $\phi \vee \phi \limplies \phi$,
	\item[PL4.] $(\phi \limplies \psi) \limplies (\chi \vee \phi \limplies \chi \vee \psi)$,
	\item[PL5.] $\bot \limplies \phi$,
	\item[Appl.]  $s:(\phi \limplies \psi ) \limplies (t:\phi  \limplies (s\tapp t):\psi )$, 
	\item[Sum.]  $s:\phi \vee t:\phi \limplies (s\ttsum t):\phi$%, $s:\phi \limplies (t\ttsum s):\phi $,
	\item[jT.] $t:\phi \rightarrow \phi $,
	\item[j4.] $t:\phi \rightarrow !t:t:\phi $.
\end{description}

The only rule of inference of $\LP_\emptyset$ is:
\begin{description}
	\item[MP.] \textit{Modus Ponens}, 
	\[\frac{\phi  \quad \phi  \limplies \psi }{\psi } \]
%	\item[AN.] \textit{Axiom Necessitation}, \[\frac{}{\justConsThree:\phi }\]
%	where  $c$ is a proof constant and $\phi $ is an axiom instance of $\JL$.
\end{description}

Given a justification logic $\JL$, a \textit{constant specification} $\CS$ for $\JL$ is a set of formulas of
the form $\justConsThree:\phi $, where $c$ is a proof constant and $\phi $ is an axiom instance of $\JL$. A constant specification $\CS$ is called \textit{axiomatically appropriate} provided for every axiom instance $\phi$  of $\JL$ there exists a proof constant $c$ such that $\justConsThree:\phi \in \CS$. The total constant specification $\mathsf{TCS}$ is defined as follows:
\[ \mathsf{TCS} = \{ \justConsThree:\phi \mid \mbox{$c$ is a proof constant and $\phi $ is an axiom instance} \}.\]
Note that $\mathsf{TCS}$ is axiomatically appropriate.

Given a constant specification $\CS$ for $\JL$, the justification logic $\JL_\CS$ is an extension of $\JL_\emptyset$ that has the formulas of $\CS$ as extra axioms. From now on when we write $\JL_\CS$ we mean that $\CS$ is a constant specification for $\JL$.

Non-empty constant specifications are used to show the internalization property for the logic of proofs. This property simulates the necessitation rule from normal modal logic.

\begin{lemma}[Internalization] \label{lem: Internalization}
	Suppose that $\CS$ is an axiomatically appropriate constant specification. If $\vdash_{\LP_\CS} \phi$, then there is a term $t \in \Terms$ such that $\vdash_{\LP_\CS} t:\phi$.
\end{lemma}
\begin{proof}
	The proof is by induction on the derivation of
	$\phi$. We have two base cases:
		\begin{itemize}
			\item If $\phi$ is an axiom instance of $\LP_\emptyset$. Then, since $\CS$ is axiomatically appropriate, there is a proof constant $\justConsThree$ such that $\justConsThree:\phi \in \CS$. Thus, put $t:= \justConsThree$.
			
			\item If $\phi = \justConsThree:\psi \in \CS$, then using axiom j4 and MP we get $\vdash_{\LP_\CS} !\justConsThree : \justConsThree : \psi$. Thus, put $t:= !\justConsThree$.
		\end{itemize}
	For the induction step suppose $\phi$ is obtained by Modus Ponens from $\psi \limplies \phi$ and $\psi$. By the induction hypothesis, there are terms $r$
		and $s$ such that $\vdash_{\LP_\CS} r:(\psi \limplies \phi)$ and $\vdash_{\LP_\CS} s:\psi$. Then put $t: = r \cdot s$ and use the axiom jK to obtain $\vdash_{\LP_\CS} r \cdot s: \phi$.
\qed	
\end{proof}

\begin{lemma}[Lifting] \label{lem: Lifting}
	Suppose that $\CS$ is an axiomatically appropriate constant specification. If 
	$$\psi_1, \ldots, \psi_n \vdash_{\LP_\CS} \phi,$$
	then there is a term $t(\vec{\justVarOne}) \in \Terms$ such that %
	$$\justVarOne_1: \psi_1, \ldots, \justVarOne_n: \psi_n \vdash_{\LP_\CS} t(\vec{\justVarOne}):\phi,$$
	where $\vec{\justVarOne}$ denotes $\justVarOne_1, \ldots, \justVarOne_n$.
\end{lemma}
\begin{proof}
	The proof is similar to the proof of Lemma \ref{lem: Internalization}. The only new case is that $\phi = \psi_i$, for some $1 \leq i \leq n$. In this case put $t := \justVarOne_i$.
	\qed	
\end{proof}

\begin{example}\label{ex: example of internalization}
	Here is a proof of $p \limplies p$ in $\LP_\emptyset$. In the following proof, suppose that $p \vee p \limplies p$,  $p \limplies p \vee p$, and $((p \vee p) \limplies p) \limplies ((p \limplies (p \vee p)) \limplies (p \limplies p))$ are axiom instances of propositional logic.
	\begin{enumerate}
		\item $((p \vee p) \limplies p) \limplies ((p \limplies (p \vee p)) \limplies (p \limplies p))$, \hfill axiom instance PL4
		
		\item $p \vee p \limplies p$, \hfill axiom instance PL3
		
		\item $p \limplies p \vee p$, \hfill axiom instance PL1
		
		\item $(p \limplies (p \vee p)) \limplies (p \limplies p)$, \hfill from  1 and 3 by MP
		
		\item $p \limplies p$. \hfill from 2 and 4 by MP
	\end{enumerate}
	Using the proof of the internalization lemma, we assign a proof term to each line of the above proof. For $\justConsOne, \justConsTwo, \justConsThree \in \CTerms$ we have:
	\begin{description}
		\item[$1'$.] $\justConsOne: [((p \vee p) \limplies p) \limplies ((p \limplies (p \vee p)) \limplies (p \limplies p))]$, \hfill from $\mathsf{TCS}$
		
		\item[$2'$.] $\justConsTwo: (p \vee p \limplies p)$, \hfill from $\mathsf{TCS}$
		
		\item[$3'$.] $\justConsThree: (p \limplies p \vee p)$, \hfill from $\mathsf{TCS}$

		\item[$4'$.] $\justConsOne \tapp \justConsTwo: [(p \limplies (p \vee p)) \limplies (p \limplies p)]$, \hfill from  1 and 3 by MP
		
		\item[$5'$.] $(\justConsOne \tapp \justConsTwo) \tapp \justConsThree: (p \limplies p)$. \hfill from 2 and 4 by MP
	\end{description}
	Thus, we showed that $\vdash_{\LP_\mathsf{TCS}} (\justConsOne \tapp \justConsTwo) \tapp \justConsThree: (p \limplies p)$. Note that the proof term $t = (\justConsOne \tapp \justConsTwo) \tapp \justConsThree$ includes all the information that exist in the proof of $p \limplies p$. Three proof constants $\justConsOne, \justConsTwo, \justConsThree$ in $t$ show that three axiom instances are used in the proof (see the axioms used in the steps $1'$--$3'$ above), and the two application operators ``$\tapp$'' in $t$ show that two applications of the rule MP are used in the proof.
\end{example}

In the remaining of this section, we recall a well known semantics  for the logic of proofs were originally introduced by Mkrtychev in \cite{Mkr97LFCS} (see also the basic modular models of \cite{KuzStu12AiML} for the logic of proofs).

\begin{definition}\label{def: LP-models}
	An $\LP_\CS$ model $\M=(\E, \V)$ consists of an evidence function $\E:\Terms \times \Formulae \to \{0,1\}$ and a valuation $\V: \Prop \to \{0,1\}$. The valuation $\V$ is extended to all formulas $\VExt: \Formulae \rightarrow \{0,1\}$ as follows:
	\begin{align*}
		\VExt(p) &= \V(p), \\
		\VExt(\bot) &= \AbotElement, \\
		\VExt(\neg \phi ) &=  1 - \VExt(\phi), \\
		\VExt(\phi \vee \psi) &= \max (\VExt(\phi),  \VExt(\psi)), \\
		\VExt(t: \phi) &= \E(t,\phi).
	\end{align*}
	The evidence function $\E$ should satisfy the following conditions:
	\begin{description}
		\item[Appl.] $\min (\E(s, \phi \limplies \psi), \E(t, \phi)) \leq \E(s \tapp t, \psi)$,
		
		\item[Sum.] $\max (\E(s,\phi), \E(t,\phi)) \leq \E(s \ttsum  t, \phi)$,
		
		\item[jT.] $\E(t, \phi) \leq \VExt(\phi)$,
		
		\item[j4.] $\E(t, \phi) \leq \E(!t, t:\phi)$,
		
		\item[CS] $\E(c, \phi) = \AtopElement$, for $\justConsThree:\phi \in \CS$. 
	\end{description}
	 A formula $\phi$ is $\LP_\CS$-valid, denoted by $\models_{\LP_\CS} \phi$, if $\VExt(\phi) = 1$ for every $\LP_\CS$ model $\M=(\E, \V)$.
\end{definition}

%Informally, $\E(t, \phi) = 1$ could be read as ``term $t$ is an evidence for formula $\phi$." 

The proof of soundness and completeness theorems is given in \cite{Mkr97LFCS} and \cite{KuzStu12AiML}. 

\begin{theorem}\label{thm: Completeness LP models}
	Let $\CS$ be a constant specification for $\LP$. Then, $\vdash_{\LP_\CS} \phi$ if{f} $\models_{\LP_\CS} \phi$.
\end{theorem}

We give the preliminary definitions of Boolean, modal, and regular algebras in the next section. Then, in Section \ref{sec: LP-CS algebras} we present algebraic semantics for $\LP$. 
%%%%%%%%%%%%%%%%%%%%%%%%%%
\section{Regular algebras}\label{sec: Regular alg}

Throughout the paper we only consider classical logics, and so we deal with Boolean algebras (see \cite{Ono-Book-2019} for a more detailed exposition). A Boolean algebra $(A, \AbotElement, \Anot, \Ajoin )$ (where the constant $\AbotElement$ is the least element and the operators $\Anot$ and $\Ajoin$ give the complement of an element and join of two elements respectively) is a structure in which the following laws hold (here $\AtopElement := \Anot \AbotElement$ is the greatest element, and $\AElementOne \Ameet \AElementTwo := \Anot((\Anot \AElementOne) \Ajoin (\Anot \AElementTwo))$ gives the meet of two elements): commutative and associative laws for join $\Ajoin$ and meet $\Ameet$, distributive laws both for meet over join and for join over meet, and the following special laws:
\begin{eqnarray*}
\AElementOne \Ajoin \AbotElement &= \AElementOne, \\\qquad
\AElementOne \Ameet \AtopElement &= \AElementOne, \\
\AElementOne \Ajoin (\Anot \AElementOne) &= \AtopElement, \\\qquad
\AElementOne \Ameet (\Anot \AElementOne) &=  \AbotElement.
\end{eqnarray*}
The implication is defined as $\AElementOne \Aimplise \AElementTwo := (\Anot \AElementOne) \Ajoin \AElementTwo$. The order $\leq$ on a Boolean algebra is defined as follows:
\begin{center}
	$\AElementOne \leq \AElementTwo$ iff $\AElementOne \Aimplise \AElementTwo = \AtopElement$ (iff $\AElementOne \Ajoin \AElementTwo = \AElementTwo $ iff $\AElementOne \Ameet \AElementTwo  = \AElementOne$).
\end{center}
Since the logic of proofs $\LP$ is a justification counterpart of modal logic $\Logic{S4}$, let us recall $\Logic{S4}$ algebras first. An $\Logic{S4}$  algebra (or an interior algebra) is a tuple $(A, \AbotElement, \Anot, \Ajoin ,  \Box)$
such that $(A, \AbotElement, \Anot, \Ajoin)$ is a Boolean algebra, and the operator $\Box : A \to A$  satisfies the following equations:
\begin{itemize}
	\item $\Box(\AtopElement) = \AtopElement$,
	
	\item $\Box(\AElementOne \Ameet \AElementTwo ) = \Box \AElementOne \Ameet \Box \AElementTwo $,
	
	\item $\Box \Box \AElementOne = \Box \AElementOne$,
	
	\item $\Box \AElementOne \Ameet \AElementOne = \Box \AElementOne$.
\end{itemize}
The completeness proof of modal logics with respect to modal algebras follows from the fact that  the regularity rule is admissible in all normal modal logics (cf. \cite{Chagrov-Zakharyaschev-1997,Blackburn-Rijke-Venema-2001}):

\[\frac{\phi \leftrightarrow \psi}{\Box \phi \leftrightarrow \Box \psi}\ Reg\]

Likewise we extend the logic of proofs with a justification counterpart of this rule. Let
${\sf HLP}_\emptyset$ be an extension of $\LP_\emptyset$ by the following justification regularity rule:

\[\frac{\phi \leftrightarrow \psi}{t:\phi \leftrightarrow t:\psi}\ JReg\]

A justification logic in which the justification regularity rule $JReg$ is admissible is called \textit{regular}. %The definition of constant specification $\CS$ for ${\sf HLP}$ and the logic ${\sf HLP}_\CS$ is similar to that given in Section \ref{sec:Justification logics}.

\begin{definition}[${\sf HLP}_\emptyset$ algebra]
	An ${\sf HLP}_\emptyset$ algebra is a tuple $\Algebra = (A, \AbotElement, \Anot, \Ajoin,  \OBox{t} )_{t \in \Terms}$ such that $(A, \AbotElement, \Anot, \Ajoin)$ is a Boolean algebra with operators $\OBox{t} : A \to A$  satisfying the following conditions. For all $\AElementOne, \AElementTwo  \in A$ and all $s,t \in \Terms$:
	\begin{description}
		\item[A-Appl.] $\OBox{s}(\AElementOne \Aimplise \AElementTwo ) \Ameet \OBox{t}(\AElementOne) \leq \OBox{s \tapp t}(\AElementTwo )$,
		
		\item[A-Sum.] $\OBox{s}(\AElementOne) \Ajoin \OBox{t}(\AElementOne) \leq \OBox{s \ttsum  t}(\AElementOne)$,
		
		\item[A-jT.] $\OBox{t}(\AElementOne) \leq \AElementOne$,
		
		\item[A-j4.] $\OBox{t}(\AElementOne) \leq \OBox{!t} (\OBox{t}(\AElementOne))$,
		
%		\item[A-CS] $\OBox{c}(\AtopElement) = \AtopElement$, 
%		\item[A-CS] $\OBox{c}(\AtopElement) = \AtopElement$ for $\justConsThree:\phi \in \CS$. 
	\end{description}
	A regular algebra is a tuple $\Algebra = (A, \AbotElement, \Anot, \Ajoin,  \OBox{t} )_{t \in \Terms}$ such that $(A, \AbotElement, \Anot, \Ajoin)$ is a Boolean algebra with operators $\OBox{t} : A \to A$  satisfying conditions A-Appl and A-Sum.\footnote{Note that every ${\sf HLP}_\emptyset$ algebra is a regular algebra.}

\end{definition}

\begin{definition}[Valuation] \label{def:valuation}
	Let $\Algebra$ be a regular algebra. A valuation on $\Algebra$ is a function $\Val : \Prop \to A$. The assignment $\ValExt : \Formulae \to A$ on $\Algebra$ is an extension of $\Val$ defined as follows
	\begin{align*}
		\ValExt(p) &= \Val(p), \\
		\ValExt(\bot) &= \AbotElement, \\
		\ValExt(\neg \phi ) &= \Anot \ValExt(\phi), \\
		\ValExt(\phi \vee \psi) &= \ValExt(\phi) \Ajoin \ValExt(\psi), \\
		\ValExt(t: \phi) &= \OBox{t}(\ValExt(\phi)).
	\end{align*}

\end{definition}

Note that from the above conditions one can obtain the following:
$$	\ValExt(\phi \limplies \psi) = \ValExt(\phi) \Aimplise \ValExt(\psi). $$

Next we give a general definition for validity which will be used throughout the paper.

\begin{definition}[Validity]\label{def:validity}
	Let $\Algebra$ be an algebra and let $\nabla$ be a set of certain \textit{distinguished elements} of $A$. The set $(\Algebra, \nabla)$ is called a matrix. If $\Algebra$ is a regular algebra, then $(\Algebra, \nabla)$ is called a regular matrix. %Let $\mathcal{C}$ be a class of matrices and $(\Algebra, \nabla) \in \mathcal{C}$.
	\begin{itemize}
%		\item A formula $\phi$ is true in the algebra $\Algebra$ under the valuation $\theta$, denoted by $\Algebra \models_\theta \phi$, if $\OValExt{\phi} \in \nabla$.
		
		\item A formula $\phi$ is true in the matrix $(\Algebra, \nabla)$, denoted by $(\Algebra, \nabla) \models \phi$, if $\OValExt{\phi} \in \nabla$ for every valuation $\Val$ on $\Algebra$.
		
		\item A formula $\phi$ is valid in a class $\mathcal{C}$ of matrices, denoted by $\mathcal{C} \models \phi$, if $(\Algebra, \nabla) \models \phi$ for every algebra $(\Algebra, \nabla) \in \mathcal{C}$.
		
		\item A logic $\mathsf{L}$ is characterized by a class $\mathcal{C}$ of matrices if for every formula $\phi$: $\vdash_\mathsf{L} \phi$ if{f} $\mathcal{C} \models \phi$.
		
		%\item Given a set of formulas $\Gamma$, a formula $\phi$ entails from $\Gamma$ in the class $\mathcal{C}$, denoted by $\Gamma \models_\mathcal{C} \phi$, provided for every algebra $(\Algebra, \nabla) \in \mathcal{C}$ such that $(\Algebra, \nabla) \models \psi$, for every $\psi \in \Gamma$, we have $(\Algebra, \nabla) \models \phi$.
	\end{itemize}
The class of all ${\sf HLP}_\emptyset$ algebras with singleton $\nabla$ is denoted by $\Algebra^N$ (the superscript $N$ denotes the use of singleton $\nabla$ for all algebras in the class).

Notation: We shall often deal with matrices $(\Algebra, \nabla)$ in which $\nabla = \{ \AtopElement_\Algebra \}$, where $\AtopElement_\Algebra$ denotes the unit of $\Algebra$. In this case instead of $(\Algebra, \nabla) \models \phi$ we write $\Algebra \models \phi$. Thus, $\Algebra \models \phi$ means that $\ValExt(\phi) = \AtopElement_\Algebra$, for every valuation $\Val$ on $\Algebra$.
\end{definition}

Now we use the so called Tarski-Lindenbaum algebra to prove  completeness. 

\begin{definition}\label{def: Tarski-Lindenbaum algebra for HLP}
	For $\phi \in \Formulae$, let
	\[ 
	\OCl{\phi} := \{ \psi \mid \vdash _{\HLP_\emptyset}\phi \liff \psi \}.
	\]
	and let
	\[
	\OCl{\Formulae} := \{ \OCl{\phi} \mid \phi \in \Formulae \}.
	\]
	The Tarski-Lindenbaum algebra
		\[
	\Algebra_{\HLP_\emptyset} := ( \OCl{\Formulae}, \AbotElement, \Anot, \Ajoin, \OBox{t} )_{t \in \Terms}
	\]
	 for $\HLP_\emptyset$ is defined as follows:
	\begin{align*}
		%\OCl{\Formulae} &:= \{ \OCl{\phi} \mid \phi \in \Formulae \} \\
		\AbotElement  &:= \OCl{\bot}, \\
		\Anot \OCl{\phi}  &:= \OCl{\neg \phi}, \\
		\OCl{\phi} \Ajoin \OCl{\psi} &:= \OCl{\phi \vee \psi}, \\
		\OBox{t}(\OCl{\phi}) &:= \OCl{t:\phi}.
	\end{align*}

	Let $\nabla = \{ \OCl{\top} \}$.
	\end{definition}

Observe that 
\[
\vdash \phi \limplies \psi \quad \Leftrightarrow \quad \OCl{\phi} \leq \OCl{\psi},
\]
and
\[
\OCl{\phi} \Aimplise \OCl{\psi} = \OCl{\phi \limplies \psi}.
\]

\begin{lemma}\label{lem: Tarski-Lindenbaum algebra is an sf HLP algebra}
	The Tarski-Lindenbaum algebra $\Algebra_{\HLP_\emptyset}$ is an ${\sf HLP}_\emptyset$ algebra.
\end{lemma}
\begin{proof}
We show that the operators $\Box_{t}$, for $t \in \Terms$, are well-defined. Checking the rest of the conditions is straightforward.

Let $t \in \Terms$ be a fixed term. If $\OCl{\phi} = \OCl{\psi}$, then $\vdash _{\HLP_\emptyset}\phi \liff \psi$. By the rule $JReg$, $\vdash _{\HLP_\emptyset} t:\phi \liff t:\psi$, and hence  $\OCl{t:\phi} = \OCl{t:\psi}$. Therefore, $\OBox{t}(\OCl{\phi}) = \OBox{t}(\OCl{\psi})$. \qed
\end{proof}

\begin{theorem}[Soundness and completeness]\label{thm: completeness HLP}
	$\vdash_{\HLP_\emptyset} \phi$ iff $\Algebra^N \models \phi$.
\end{theorem}
\begin{proof}
	Soundness is straightforward. For completeness define the valuation $\Val$ as follows
	\[
	\OVal{p} := \OCl{p}, \qquad \text{ for } p \in \Prop.
	\]
Now it is easy to prove the Truth Lemma: for every formula $\phi$ 
		\[
		\OValExt{\phi} = \OCl{\phi}.
		\]
The proof is by induction on the complexity of $\phi$. The proof of the case $\phi = t:\psi$ follows from Definition \ref{def:valuation} and Definition \ref{def: Tarski-Lindenbaum algebra for HLP}, and the proof of other cases are standard.

Finally, completeness follows easily from the Truth Lemma. If $\not \vdash_{\HLP_\emptyset} \phi,$ then $\ValExt(\phi)=\OCl{\phi} \not \in \nabla = \{ \OCl{\top} \}$, and hence $(\Algebra_{\HLP_\emptyset},\nabla) \not \models \phi$. 		
\qed

\end{proof}

 Theorem~\ref{thm: completeness HLP} shows that ${\sf HLP}_\emptyset$ is characterized by the class $\Algebra^N$, i.e. the class of all ${\sf HLP}_\emptyset$ matrices with singleton $\nabla.$ 
  
\begin{theorem}\label{thm: characterization of singleton nabla}
	If a justification logic $\JL$ is characterized by a class of regular matrices $\mathcal{C}$ with singleton $\nabla$ then the rule $JReg$ is admissible in $\JL$.
\end{theorem}
\begin{proof}
	Suppose that $\vdash_\JL \phi \liff \psi$. Then, $\mathcal{C} \models \phi \liff \psi$. Thus, for every $\Algebra \in \mathcal{C}$ and for every valuation $\Val$ on $\Algebra$, we have $\OValExt{\phi} = \OValExt{\psi}$. Thus, $\OBox{t}(\OValExt{\phi}) = \OBox{t}(\OValExt{\psi})$, for every $\Val$. Hence, $\OValExt{t:\phi} = \OValExt{t:\psi}$, for every $\Val$. Thus, $\mathcal{C} \models t:\phi \liff t:\psi$, and hence $\vdash_\JL t:\phi \liff t:\psi$. \qed
\end{proof}

Let us consider  the rule $JReg$. This rule says that if $\phi \liff \psi$ is a theorem, then every proof $t$ of $\phi$ is a proof of $\psi$ and vise versa. For instance, let $\phi$ be $p \limplies p$, and $\psi$ be $\bot \limplies q$. We know that $\phi \liff \psi$ is a theorem of propositional logic. It should be obvious that in this example, a proof of $\phi$ could be different from that of $\psi$ and it is not the case that every proof of $\phi$ is a proof of $\psi$ and vise versa. For example, both $(\justConsOne \tapp \justConsTwo) \tapp \justConsThree: (p \limplies p)$ (see Example \ref{ex: example of internalization}) and $\mathsf{d}:(\bot \limplies q)$, for some $\mathsf{d} \in \CTerms$, are provable in $\LP_\mathsf{TCS}$ but $(\justConsOne \tapp \justConsTwo) \tapp \justConsThree$ and $\mathsf{d}$ are different proof terms, and one can show that $(\justConsOne \tapp \justConsTwo) \tapp \justConsThree: (\bot \limplies q)$ is not a theorem of $\LP_\mathsf{TCS}$.\footnote{Even since $\phi$ is a theorem of paraconsistent logics but $\psi$ is not, it seems that the `informal justifications' of $\phi$ and $\psi$ are different too.} %For example, let $\phi$ be $\neg \chi \vee \chi$, and $\psi$ be $(\chi_1 \limplies \chi_2) \limplies (\neg \chi_1 \vee \chi_2)$. We know that $\phi \liff \psi$ is a theorem of propositional logic. It should be obvious that in this case a proof of $\phi$ is different from that of $\psi$.\footnote{Even since $\psi$ is a theorem of the intuitionistic logic but $\phi$ is not, it seems that the informal justifications of $\phi$ and $\psi$ are different too.}

It is noteworthy that a mono-agent version of $JReg$ is used in the axiomatic formulation of some of the relevant justification logics of \cite{Standefe-IGPL-2019}. One might argue that for an equivalence to be valid in a relevant logic there must be some connection between the two equivalent propositions, and thus $JReg$ is expected to be admissible in this setting. Nonetheless, it is not still clear why two equivalent propositions should be known for the same reason. It seems that an argument in favor of $JReg$ in a relevant logic depends on a sensible notion of `connection', and thus this issue has yet to be investigated more.

Since $JReg$ is not admissible in the standard justification logics, in order to present algebraic semantics for the logic of proofs $\LP$ we consider two possibilities: 
\begin{enumerate}
	\item $\nabla$ is not singleton. 
	
	\item  Replacing the operators $\OBox{t} : A \to A$ with an alternate.    
\end{enumerate}

In the rest of this section we consider the first possibility. We show that the logic of proofs $\LP_\emptyset$ is characterized by the class of all ${\sf HLP}_\emptyset$ matrices when $\nabla$ is not singleton. This class is denoted by $\Algebra^{\sf Inf}$.

\begin{definition}\label{def: Tarski-Lindenbaum algebra with non-singleton nabla}
	The Tarski-Lindenbaum algebra
	\[
	\Algebra_{\LP_\emptyset}^{\sf Inf} := ( \OCl{\Formulae}, \AbotElement, \Anot, \Ajoin,  \OBox{t} )_{t \in \Terms},
	\]
	%
%	where
%	%
%	\[
%	\OCl{\Formulae} := \{ \OCl{\phi} \mid \phi \in \Formulae \},
%	\]
	%
	is defined similar to Definition \ref{def: Tarski-Lindenbaum algebra for HLP} with the difference that now the class of a formula is defined as follows:
	\[
	\OCl{\phi} := \{ \psi \mid (\forall n \geq 0) (\forall s_1, \ldots, s_n \in \Terms) \vdash_{\LP_\emptyset} s_n : \cdots : s_1 : \phi \liff s_n : \cdots : s_1 : \psi \}.
	\]
%	
%	
%	\[
%	\OBox{t}(\OCl{\phi}) = \OCl{t:\phi}
%	\]
	Let $\nabla' =  \{ \OCl{\phi} \mid \vdash_{\LP_\emptyset} \phi \}$.
\end{definition}

%\textbf{Question:} What are the members of $\OCl{\phi}$? It seems that $\OCl{\phi} = \{ \phi \}$!!!

\begin{lemma}
	The Tarski-Lindenbaum algebra $\Algebra_{\LP_\emptyset}^{\sf Inf}$ is an ${\sf HLP}_\emptyset$ algebra.
\end{lemma}
\begin{proof}
%Similar to the proof of Lemma~\ref{lem: Tarski-Lindenbaum algebra is an sf HLP algebra}.
We show that the operators $\Box_{t}$, for $t \in \Terms$, are well-defined. Checking the rest of the conditions is straightforward.

Let $t \in \Terms$ be a fixed term. If $\OCl{\phi} = \OCl{\psi}$, then by the definition of the class of formulas in Definition \ref{def: Tarski-Lindenbaum algebra with non-singleton nabla} we get for all $n \geq 0$ and for all $s_1, \ldots, s_n \in \Terms$:
\[
\vdash_{\LP_\emptyset} s_n : \cdots : s_1 : t: \phi \liff s_n : \cdots : s_1 : t: \psi.
\]

Hence  $\OCl{t:\phi} = \OCl{t:\psi}$. Therefore, $\OBox{t}(\OCl{\phi}) = \OBox{t}(\OCl{\psi})$. \qed
\end{proof}

\begin{theorem}[Soundness and completeness]\label{thm: completeness regular algebra with infinite nabla}
	$\vdash_{\LP_\emptyset} \phi$ iff $\Algebra^{\sf Inf} \models \phi$.
\end{theorem}
\begin{proof}
Similar to the proof of Theorem~\ref{thm: completeness HLP}. Just note that if $\not \vdash_{\LP_\emptyset} \phi,$ then $\ValExt(\phi)=\OCl{\phi} \not \in \nabla'$, and hence $(\Algebra_{\LP_\emptyset}^{\sf Inf},\nabla') \not \models \phi$. \qed
\end{proof}

It is worth nothing that considering $\nabla'$ as the set of all theorems is not a viable solution. In fact this theorem conveys nothing but the fact that $\LP_\emptyset$ is closed under substitution (cf. \cite[page 195]{Chagrov-Zakharyaschev-1997}). Hence in the rest of the paper, we study the second case where the operators $\OBox{t}$ on the algebra are replaced by functions on the formulas.
%\textbf{Question:} This completeness result contradicts Theorem \ref{thm: characterization of singleton nabla}!!!
%%%%%%%%%%%%%%%%%%%%%%%%%%%%%%%%
\section{Full $\LP_\CS$ algebras}
\label{sec: LP-CS algebras}

As mention before, the rule $JReg$ is not plausible in justification logics.\ However, since this rule is essential to have well-defined operators in the Tarski-Lindenbaum algebras, in this section (and also the next sections) by considering functions on formulas instead of operators on the algebra we overcome this issue. %In fact, we define an $\LP_\emptyset$ algebra as a Boolean algebra enriched with an infinite number of relations.

\begin{definition}\label{def: LP-CS algebra}
	A pre $\LP_\CS$ algebra is a tuple $\Algebra = (A, \AbotElement, \Anot, \Ajoin, \OBox{t})_{t \in \Terms}$ such that $(A, \AbotElement, \Anot, \Ajoin)$ is a Boolean algebra, and for each $t \in \Terms$,   $\OBox{t} : \Formulae \to A$ is a function satisfying the following conditions. For all $\phi, \psi \in \Formulae$, all $s,t \in \Terms$, and all $\justConsThree \in \CTerms$:
	\begin{description}
		\item[Al-Appl-$\LP_\CS$.] $\OBox{s}(\phi \limplies \psi) \Ameet \OBox{t}(\phi) \leq \OBox{s \tapp t}(\psi)$,
		
		\item[Al-Sum-$\LP_\CS$.] $\OBox{s}(\phi) \Ajoin \OBox{t}(\phi) \leq \OBox{s \ttsum  t}(\phi)$,
		
		%\item[Al-jT.] $\OBox{t}(\phi) \leq \ValExt(\phi)$,
		
		\item[Al-j4-$\LP_\CS$.] $\OBox{t}(\phi) \leq \OBox{!t} (\jbox{t} \phi)$,
		
		\item[Al-CS.] $\OBox{\justConsThree}(\phi) = \AtopElement$, for $\justConsThree:\phi \in \CS$.
		
	\end{description}

	%	The class of all $\LP_\CS$ algebras is denoted by $\Algebra_{\LP_\CS}$.
	%	Let $\nabla$ be a set of certain \textit{distinguished elements} of $A$.
	%	The class of all $\LP_\CS$ algebras with singleton $\nabla$ is denoted by $\Algebra_{\LP_\CS}^N$.
\end{definition}

Let $T$ be a set.\ In what follows we  call $\Algebra = (A, \AbotElement, \Anot, \Ajoin,  \OBox{t} )_{t \in T}$ an $\LP_\CS$ algebra \textit{over} $T$. One may justify this terminology by considering an $\LP_\CS$ algebra $\Algebra = (A, \AbotElement, \Anot, \Ajoin,  \OBox{t} )_{t \in T}$ as a Boolean algebra $(A, \AbotElement, \Anot, \Ajoin)$  endowed with a family of relations $\OBox{t}$ parametrized by elements $t \in T$. In the next sections, we consider various algebraic structures on the set $T$. 

\begin{definition}\label{def: assignment of full LP-CS algebras}
	Given a valuation  $\Val : \Prop \to A$, the assignment $\ValExt : \Formulae \to A$ on $\Algebra$ is defined as in Definition \ref{def:valuation} with the following difference:
	\begin{align*}
		\ValExt(t: \phi) &= \OBox{t}(\phi).
	\end{align*}
	Validity is defined as in Definition \ref{def:validity}.
\end{definition}

\begin{definition}
	A full $\LP_\CS$ algebra $\Algebra = (A, \AbotElement, \Anot, \Ajoin, \OBox{t})_{t \in \Terms}$ is a pre $\LP_\CS$ algebra that satisfies the following condition. For all $\phi \in \Formulae$, all $t \in \Terms$, and  all valuation $\Val: \Prop \to A$:
	\begin{description}
		\item[Al-jT-$\LP_\CS$.] $\OBox{t}(\phi) \leq \ValExt(\phi)$.
		
	\end{description}
	The class of all full $\LP_\CS$ algebras with singleton $\nabla  = \{\AtopElement\}$ is denoted by $\Algebra_{\LP_\CS}^{\sf full}$.
\end{definition}

\begin{example}\label{ex: full LP algebra of truth values}
	Let $\textbf{2} = (\{ \AbotElement, \AtopElement \}, \AbotElement, \dot{-}, \max)$ be the Boolean algebra of truth values, where $\dot{-}$ is defined by $\dot{-}\AElementOne = \AtopElement - \AElementOne$. Here, the join of two elements is given by maximum ($\max$), and it is easy to show that the meet of two elements is given by minimum ($\min$). Given a constant specification $\CS$ for $\LP$, we construct a full $\LP_\CS$ algebra based on the Boolean algebra ${\bf 2}$. For every $t \in \Terms$, define $\Box_{t}$ by induction on the complexity of $t$ as follows. For every $\phi \in \Formulae$:
	\begin{align*}
		\OBox{\justVarOne} (\phi) &:= 0, \\
		\OBox{\justConsThree} (\phi) &:= 
		\begin{cases}
			1 &\mbox{if $\justConsThree:\phi \in \CS$,}  \\
			0 & \mbox{otherwise. } 
		\end{cases}
		\\
		\OBox{s \tapp t} (\phi) &:= \max \{ \min(\OBox{s}(\psi \limplies \phi), \OBox{t}(\psi)) \mid \psi \in \Formulae  \}, \\
		\OBox{s \ttsum t} (\phi)&:= \max (\OBox{s}(\phi), \OBox{t}(\phi)) \\
		\OBox{\tinspect t} (\phi) &:= 
		\begin{cases}
			\OBox{t}(\psi)  &\mbox{if $\phi = t:\psi$,}  \\
			0 & \mbox{otherwise. } 
		\end{cases}
	\end{align*}
	Let ${\bf 2}_{\LP_\CS} = (\{ \AbotElement, \AtopElement \}, \AbotElement, \dot{-}, \max, \OBox{t} )_{t \in \Terms}$. It is easy to show that ${\bf 2}_{\LP_\CS}$ is a pre $\LP_\CS$ algebra. Now we show that ${\bf 2}_{\LP_\CS}$ is a full $\LP_\CS$ algebra. To this end, by induction on the complexity of the term $t$ we show that:
	\[
	\OBox{t}(\phi) \leq \ValExt(\phi).
	\]
	We only check the case $t = s \tapp r$ (the proof for other cases is simpler). By the induction hypothesis, we have
	\[
	\OBox{s}(\psi \limplies \phi) \leq \ValExt(\psi \limplies \phi) \quad
	\mbox{ and }\quad
	\OBox{t}(\psi) \leq \ValExt(\psi).
	\]
	Thus, 
	\[
	\min(\OBox{s}(\psi \limplies \phi), \OBox{t}(\psi))  \leq
	\min (\ValExt(\psi \limplies \phi), \ValExt(\psi)),
	\]
	and hence 
	\[
	\min(\OBox{s}(\psi \limplies \phi), \OBox{t}(\psi))  \leq \ValExt(\phi).
	\]
	By taking maximum over all formulas $\psi$, we get
	\[
	\OBox{s \tapp t} (\phi) \leq \ValExt(\phi).
	\]
\end{example}

\begin{lemma}
	The regularity rule
	\[\frac{\phi \leftrightarrow \psi}{t:\phi \leftrightarrow t:\psi}\ JReg\]
	is not validity preserving in the class of full $\LP_\CS$ algebras, i.e. there exists $\phi, \psi \in \Formulae$ and $t \in \Terms$ such that $\Algebra_{\LP_\CS}^{\sf full} \models \phi \liff \psi$ but $\Algebra_{\LP_\CS}^{\sf full} \not \models t: \phi \liff t:\psi$.
\end{lemma}
\begin{proof}
	Let $\CS = \{ \justConsThree : (p \wedge p \limplies p) \}$ be a  constant specification for $\LP$. Consider the full $\LP_\CS$ algebra ${\bf 2}_{\LP_\CS}$ from Example \ref{ex: full LP algebra of truth values}.
	It is easy to show that $\Algebra_{\LP_\CS}^{\sf full} \models (p \wedge p \limplies p) \liff \top$. On the other hand, we observe that 
	$$\OBox{c} (p \wedge p \limplies p) = 1 \quad \text{ and } \quad \OBox{c} (\top) = 0.$$
	Thus, given an arbitrary valuation $\theta$, we have $\ValExt (\justConsThree : (p \wedge p \limplies p)) \neq \ValExt (\justConsThree : \top)$. Hence, $\Algebra_{\LP_\CS}^{\sf full} \not \models \justConsThree : (p \wedge p \limplies p) \liff \justConsThree : \top$. \qed
\end{proof}

%In what follows, we prove completeness theorems with respect to the classes $\Algebra_{\LP_\CS}^{\sf full}$  and $\Algebra_{\LP_\CS}^{\sf 2}$, and we also show a representation theorem for $\LP_\CS$.
In what follows, we prove completeness theorem with respect to the classes $\Algebra_{\LP_\CS}^{\sf full}$ and we show a representation theorem for $\LP_\CS$.
%%%%%%%%%%%%%%%%%%%%%%%%%%%%%%%%%%%%%%%%%%%%%%%%%%%%
\subsection{Completeness of $\LP_\CS$}

Completeness is proved by the standard method of defining the Tarski-Lindenbaum algebra for $\LP_\CS$.

\begin{definition}\label{def: Tarski-Lindenbaum algebra for LP-CS}
	For $\phi \in \Formulae$, let
	\[
	\OCl{\phi} := \{ \psi \mid \vdash_{\LP_\CS} \phi \liff \psi \},
	\]
	and let
	\[
	\OCl{\Formulae} := \{ \OCl{\phi} \mid \phi \in \Formulae \}.
	\]
	The Tarski-Lindenbaum algebra
	\[
	\Algebra_{\LP_\CS}^f := ( \OCl{\Formulae}, \AbotElement, \Anot, \Ajoin, \OBox{t})_{t \in \Terms},
	\]
	for $\LP_\CS$ is defined as follows:
	\begin{align*}
		%\OCl{\Formulae} &:= \{ \OCl{\phi} \mid \phi \in \Formulae \} \\
		\AbotElement  &:= \OCl{\bot}, \\
		\Anot \OCl{\phi}  &:= \OCl{\neg \phi}, \\
		\OCl{\phi} \Ajoin \OCl{\psi} &:= \OCl{\phi \vee \psi}, \\
		\OBox{t}(\phi) &:= \OCl{t:\phi}.
		%\OVal{p} &:= \OCl{p} \qquad p \in \Prop.
	\end{align*}
	Note that $\AtopElement = \OCl{\top}$.
	%	Let $\nabla = \{ \OCl{\top} \}$.
\end{definition}

\begin{lemma}
	The Tarski-Lindenbaum algebra $\Algebra_{\LP_\CS}^f$ is a pre $\LP_\CS$ algebra.
\end{lemma}
\begin{proof}
	The proof is straightforward. We only verify the condition $Al$-$CS$. Suppose that $\justConsThree:\phi \in \CS$. Then,
	$
	\OBox{\justConsThree}(\phi) = \OCl{\justConsThree:\phi} = \OCl{\top} = \AtopElement.
	$ \qed
\end{proof}

\begin{theorem}[Soundness and completeness]\label{thm: completeness full LP-CS algebras}
	$\vdash_{\LP_\CS} \phi$ iff $\Algebra_{\LP_\CS}^{\sf full} \models \phi$.
\end{theorem}
\begin{proof}
	Similar to the proof of Theorem \ref{thm: completeness HLP}. Define the valuation $\Val$ as follows
	\[
	\OVal{p} := \OCl{p}, \qquad \text{ for } p \in \Prop.
	\]
	Now it is easy to prove the Truth Lemma: for every formula $\phi$ 
	\[
	\OValExt{\phi} = \OCl{\phi}.
	\]
	The proof is by induction on the complexity of $\phi$. The proof of the case $\phi = t:\psi$ is as follows: 
	$$\tilde{\theta}(t:\psi)=\OBox{t}(\psi)=[t:\psi].$$
	Next we show that the Tarski-Lindenbaum algebra $\Algebra_{\LP_\CS}^f$ satisfies $Al$-$jT$-$\LP_\CS$, and hence is a full $\LP_\CS$ algebra:
	\[
	\OBox{t}(\phi) = \OCl{t:\phi} \leq \OCl{\phi} = \ValExt(\phi).
	\]
	Finally, completeness follows easily from the Truth Lemma. \qed
\end{proof}

%%%%%%%%%%%%%%%%%%%%%%%%%%%%%%%%%%%%%%%%%%%%%%%%%%%%%%%%%%%%%%%%%
\subsection{Representation theorem for full $\LP_\CS$ algebras}

In this section we show a representation theorem for full $\LP_\CS$ algebras.  We first recall that
%\begin{definition}
the power set algebra of a set $A$ is the
structure
\[
\mathfrak{P}(A) = (\powerset(A),\emptyset,\setminus,\cup),
\]
where $\powerset(A)$ is the power set of $A$, $\setminus$ is the complement relative to $A$. A set algebra is a subalgebra of a power set algebra.
%\end{definition}

We now recall a classical result in propositional logic that will be used in the sequel (cf. \cite{Blackburn-Rijke-Venema-2001}). 

\begin{theorem}[Stone Representation Theorem]
	Every Boolean algebra is isomorphic to a set algebra.
\end{theorem}

Suppose that $\Algebra$ and $\mathcal{B}$ are two full $\LP_\CS$ algebras. We say that $f$ is an isomorphism between $\Algebra$ and $\mathcal{B}$ provided that 
\[
f(\OBox{t}^\Algebra (\phi)) = \OBox{t}^\mathcal{B} (\phi),
\]
for every $\phi \in \Formulae$. 

Suppose that $\Algebra$ and $\mathcal{B}$ are isomorphic Boolean algebras via an isomorphism $f$. Then if $\Algebra$ is a full $\LP_\CS$ algebra with functions $\OBox{t}^\Algebra$, define $\OBox{t}^\mathcal{B}$ on the algebra $\mathcal{B}$ as follows: 
\begin{equation}\label{eq: Box induced relation via isomorphism LP-CS}
	\OBox{t}^\mathcal{B} (\phi) :=  f(\OBox{t}^\Algebra(\phi)). 
\end{equation}
%\begin{equation}\label{eq: Val induced relation via isomorphism LP-CS}
%\Val^\mathcal{B}(p) := f(\Val^\mathcal{A}(p))
%\end{equation}

Now we show that $\mathcal{A}$ and $\mathcal{B}$ are isomorphic as full $\LP_\CS$ algebras. 

\begin{lemma}\label{lem: prestone LP-CS}
	Suppose that $\Algebra$ and $\mathcal{B}$ are two isomorphic Boolean algebras via isomorphism $f$. Then if $\Algebra$ is a full $\LP_\CS$ algebra with functions $\OBox{t}^\Algebra$, then $\mathcal{B}$ is a full  $\LP_\CS$ algebra with functions $\OBox{t}^\mathcal{B}$ defined in \eqref{eq: Box induced relation via isomorphism LP-CS}. Moreover, $f$ is an isomorphism of full $\LP_\CS$ algebras.
\end{lemma}
\begin{proof}
		First note that since $f$ is an isomorphism between $\Algebra$ and $\mathcal{B}$, for every $\AElementOne, \AElementTwo  \in A$ we have
	\begin{equation}\label{eq: isomorphism is order preserving}
		\AElementOne \leq \AElementTwo  \quad \text{ iff } \quad f(\AElementOne) \leq f(\AElementTwo )
	\end{equation}
	In order to show that $\mathcal{B}$ is a full  $\LP_\CS$ algebra we only check the conditions  $Al$-$Appl$-$\LP_\CS$ and $Al$-$jT$-$\LP_\CS$ from Definition \ref{def: LP-CS algebra} (other cases can be verified similarly).\ 
	
	For $Al$-$Appl$-$\LP_\CS$, suppose that $\phi, \psi \in \Formulae$ and $s,t \in \Terms$:
	\begin{align*}
		&  \OBox{s}^\mathcal{B}(\phi \limplies \psi) \Ameet  \OBox{t}^\mathcal{B}(\phi)=\\
		&  f(\OBox{s}^\mathcal{A}(\phi \limplies \psi)) \Ameet  f(\OBox{t}^\mathcal{A}(\phi))=\\
		& f(\OBox{s}^\mathcal{A}(\phi \limplies \psi) \Ameet  \OBox{t}^\mathcal{A}(\phi))\leq \\ 
		& f( \OBox{s \tapp t}^\mathcal{A}(\psi))= \OBox{s \tapp t}^\mathcal{B}(\psi).
	\end{align*}
	
	For $Al$-$jT$-$\LP_\CS$, suppose that $\phi \in \Formulae$, $t \in \Terms$ and $\Val : \Prop \to B$:
	\[
	\OBox{t}^\mathcal{B} (\phi) =  f(\OBox{t}^\Algebra(\phi)) \leq f(f^{-1} (\ValExt(\phi))) = \ValExt(\phi).
	\]
	The above inequality follows from the fact that $f$ is an isomorphism and $f^{-1} \circ \ValExt : \Formulae \to A$ is an assignment on $\Algebra$.
	\qed	
\end{proof}

A set algebra enriched with the functions $\OBox{t}$, for all $t \in \Terms$, is called a \textit{full $\LP_\CS$ set algebra}.

\begin{theorem}[Representation Theorem]\label{thm: Representation Theorem LP-CS}
	Every full $\LP_\CS$ algebra is isomorphic to a full $\LP_\CS$ set algebra.
\end{theorem}

\begin{proof}
	Follows from the Stone representation theorem and Lemma~\ref{lem: prestone LP-CS}. \qed
\end{proof}

\begin{theorem}
	$\LP_\CS$ is complete with respect to the class of all full $\LP_\CS$ set algebras.
\end{theorem}

\begin{proof}
		If $\not \vdash \phi$, then by Theorem \ref{thm: completeness full LP-CS algebras}, there is a full $\LP_\CS$ algebra $\Algebra$ and valuation $\Val:\Prop\to A$ such that $\Algebra \not \models_\theta \phi$, i.e $\ValExt(\phi) \neq \AtopElement_\Algebra$. By the Representation Theorem \ref{thm: Representation Theorem LP-CS} there is an isomorphism $f$ and a full $\LP_\CS$ set algebra $\mathcal{B} = (B, \ldots)$ isomorphic to $\Algebra$. Define the valuation $\Val': \Prop \to B$ as follows:
	\begin{equation*}
		\Val'(p) := f(\Val(p)).
	\end{equation*}
	It is easy to prove that 
	\[
	\ValExt'(\phi) = f(\ValExt(\phi)),
	\]
	for every formula $\phi$. Finally we get
	\[
	\ValExt'(\phi) = f(\ValExt(\phi)) \neq \AtopElement_\mathcal{B}
	\]
	Thus, $\mathcal{B} \not \models_{\theta'} \phi$.
		\qed
\end{proof}

%%%%%%%%%%%%%%%%%%%%%%%%%%%%%%%%%%%%%%%%%%%%%%%%%%%%%%%%%%%%%%%%%%%

\subsection{Binary $\LP_\CS$ algebras}

The full algebra ${\bf 2}_{\LP_\CS}$ of Example \ref{ex: full LP algebra of truth values} is an example of full $\LP_\CS$ algebras whose carrier is the two elements set $\{ \AbotElement, \AtopElement \}$. It is known that propositional logic is complete with respect to the Boolean algebra of truth values $\textbf{2}$ (cf. \cite{Blackburn-Rijke-Venema-2001}). Next we define a new class of $\LP_\CS$ algebras whose Boolean part is the Boolean algebra of truth values. Then, we establish a completeness theorem with respect to this  new class of $\LP_\CS$ algebras.

% A full $\LP_\CS$ algebra with $\nabla  = \{\AtopElement\}$ whose Boolean part are defined as the Boolean algebra $\textbf{2}$ is called a binary $\LP_\CS$ algebra. Let $\Algebra_{\LP_\CS}^{\sf 2}$ denote the class of all binary $\LP_\CS$ algebras.

\begin{definition}\label{def: Binary LP-CS algebra}
	A binary $\LP_\CS$ algebra is a tuple $\Algebra = (\{ \AbotElement, \AtopElement \}, \AbotElement, \dot{-}, \max, \OBox{t}, \Val)_{t \in \Terms}$ such that the structure $(\{ \AbotElement, \AtopElement \}, \AbotElement, \dot{-}, \max)$ is the Boolean algebra defined in Example \ref{ex: full LP algebra of truth values}, for each $t \in \Terms$,   $\OBox{t} : \Formulae \to \{ \AbotElement, \AtopElement \}$ is a function, and $\Val : \Prop \to \{ \AbotElement, \AtopElement \}$ is a valuation. The function $\Val$ is extended to the assignment $\ValExt : \Formulae \to \{ \AbotElement, \AtopElement \}$ as in Definition \ref{def: assignment of full LP-CS algebras}. The functions $\OBox{t}$ should  satisfy the following conditions. For all $\phi, \psi \in \Formulae$, all $s,t \in \Terms$, and all $\justConsThree \in \CTerms$:
	\begin{description}
		\item[Al-Appl-$\LP_\CS$.] $\OBox{s}(\phi \limplies \psi) \Ameet \OBox{t}(\phi) \leq \OBox{s \tapp t}(\psi)$,
		
		\item[Al-Sum-$\LP_\CS$.] $\OBox{s}(\phi) \Ajoin \OBox{t}(\phi) \leq \OBox{s \ttsum  t}(\phi)$,
		
		%\item[Al-jT.] $\OBox{t}(\phi) \leq \ValExt(\phi)$,
		
		\item[Al-j4-$\LP_\CS$.] $\OBox{t}(\phi) \leq \OBox{!t} (\jbox{t} \phi)$,
		
		\item[Al-CS.] $\OBox{\justConsThree}(\phi) = \AtopElement$, for $\justConsThree:\phi \in \CS$.
		
		\item[Al-jT-$\LP_\CS$.] $\OBox{t}(\phi) \leq \ValExt(\phi)$.
	\end{description}
	
Let $\Algebra_{\LP_\CS}^{\sf 2}$ denote the class of all binary $\LP_\CS$ algebras with $\nabla  = \{\AtopElement\}$.	
	
\end{definition}

In order to show completeness with respect to the class $\Algebra_{\LP_\CS}^{\sf 2}$, we first show that every $\LP_\CS$ model is equivalent to a binary $\LP_\CS$ algebra, and vice versa. Then, completeness of $\LP_\CS$ with respect to $\Algebra_{\LP_\CS}^{\sf 2}$ follows from completeness of $\LP_\CS$ models. We need the following auxiliary lemma.

\begin{lemma}\label{lem: M-model and binary algebra are equivalent} 
	Let  $\M = (\E, \V)$ be an $\LP_\CS$ model  and $\Algebra = (\{ \AbotElement, \AtopElement \}, \AbotElement, \dot{-}, \max, \OBox{t}, \Val)_{t \in \Terms}$ be a binary $\LP_\CS$ algebra. If $\V = \Val$ and for all $t \in \Terms$ and all $\phi \in \Formulae$:
	\begin{equation}\label{eq: M-model and binary algebra - equivalence definition}
		\E(t, \phi) = \OBox{t}(\phi), 
	\end{equation}
	then, $\M$ and $\Algebra$ are equivalent, i.e. for all $\phi \in \Formulae$:
	\begin{equation}\label{eq: M-model and binary algebra - equivalence formula}
		\VExt(\phi) = \ValExt(\phi).
	\end{equation}
\end{lemma}
\begin{proof}
	The proof is by induction on the complexity of $\phi$. The base case follows from $\V = \Val$. For the induction step we only check the case where $\phi = t: \psi$:
	\[
	\VExt (t: \psi) = \E(t, \psi) = \OBox{t} (\psi) = \ValExt (t:\psi).
	\]
	Note that in the above equations the second equality follows from \eqref{eq: M-model and binary algebra - equivalence definition}. \qed
\end{proof}

\begin{lemma}\label{lem: M-model and binary algebra translation}
	For every $\LP_\CS$ model there is an equivalent binary $\LP_\CS$ algebra, and vice versa.
\end{lemma}
\begin{proof}
	Suppose that $\M = (\E, \V)$ is an $\LP_\CS$ model. Define the binary $\LP_\CS$ algebra $$\Algebra = (\{ \AbotElement, \AtopElement \}, \AbotElement, \dot{-}, \max, \OBox{t}, \Val)_{t \in \Terms}$$ as follows: Let $\Val := \V$, and for every $t \in \Terms$ let $\OBox{t}$ be defined on formulas as \eqref{eq: M-model and binary algebra - equivalence definition}. Then, by Lemma \ref{lem: M-model and binary algebra are equivalent}, $\M$ and $\Algebra$ are equivalent. The proof for the converse direction is similar. \qed
\end{proof}

\begin{lemma}\label{lem: M-model and binary algebra validity equivalence}
	$\models_{\LP_\CS} \phi$ iff $\Algebra_{\LP_\CS}^{\sf 2} \models \phi$.
\end{lemma}
\begin{proof}
	Follows from Lemma \ref{lem: M-model and binary algebra translation}. \qed
\end{proof}

\begin{theorem}	[Soundness and completeness]\label{thm: completeness binary LP-CS algebras}
	$\vdash_{\LP_\CS} \phi$ iff $\Algebra_{\LP_\CS}^{\sf 2} \models \phi$.
\end{theorem}
\begin{proof}
	Follows from Theorem \ref{thm: Completeness LP models} and Lemma \ref{lem: M-model and binary algebra validity equivalence}. \qed
\end{proof}
%%%%%%%%%%%%%%%%%%%%%%%%%%%%%%%%

%%%%%%%%%%%%%%%%%%%%%%%%%%%%%%%%%%%%%%%%%%%%%%%%%%%%%%%%%%%%%%%%%%%%%%%%%

\section{$\LP$-Algebras over arbitrary term Boolean algebras}
\label{sec: LP-CS algebras over arbitrary Boolean algebras}

%In this section we consider an algebraic structure on proof terms, and stipulate that proof terms constitute a Boolean algebra. 
In this section we impose a Boolean structure on proof terms of $\LP$. The resulting logic is an extension of $\LP$ and is denoted by $\LPBool$.

Proof term and formulas of $\LPBool$ are constructed by the following mutual grammars:  
\begin{gather*}
t,s  ::= \justConsThree \mid \justVarOne \mid \Abot \mid - t \mid t \ttsum s  \mid   t \tapp s \mid  \ \tinspect t  \, , \\
\phi, \psi  ::= p \mid t \lequal s \mid \bot \mid \neg \phi \mid \phi \vee \psi  \mid \jbox{t} \phi \, .
\end{gather*}
%\begin{gather*}
%t  ::= \justConsThree \mid x \mid \Abot \mid 1 \mid  t \odot t  \mid  t \ttsum t  \mid - t \mid  t \tapp t \mid  \ \tinspect t  \, , \\
%\phi  ::= p \mid t \lequal t \mid \bot \mid \phi \limplies \phi  \mid \jbox{t} \phi \, .
%\end{gather*}
%
where $\justConsThree \in \CTerms$, $\justVarOne \in \VTerms$, and $p \in \Prop$.  Define $\Atop := -\Abot$ and $s \odot t := -(- s \ttsum - t)$.  
Let $\TermsBool$ and $\FormulaeBool$ denote the set of all terms and the set of all formulas of $\LPBool$ respectively.

We briefly give a possible meaning of the Boolean operators on proof terms. Let $s \ttsum t$ prove (or justify) everything that $s$ or $t$ proves (or justifies), and let $-t$ prove (or justify) the negation of everything that $t$ proves (or justifies). For example, let $\justVarOne$ denote ``my observation of the table in front of me". Thus $\justVarOne$ justifies the proposition ``there is a table in front of me". Then, for an arbitrary $\justVarTwo$, the term $\justVarOne + \justVarTwo$ justifies again the proposition ``there is a table in front of me", and $-\justVarOne$ justifies, say, the absence of a table in front of me. As we will observe later $\Atop$ proves (or justifies) every theorem, and thus $\Abot$ proves (or justifies) the negation of theorems.% For more detail we refer the reader to Theorem \ref{thm: Generic Example of term Boolean algebra}.

Note that we use constants from $\CTerms$ as proofs or justifications for axioms, and since $\Abot$ does not have this role it is not considered as a proof constant. Moreover, we keep use $\cdot$ as a term  operator formalizing the modus ponens rule, while use $\odot$ as the dual of $\ttsum$. Here, the term operator $\ttsum$ is used for the combination of two justifications (as in the ordinary justification logics) as well as for the join operator in Boolean algebras of terms.
%%%%%%%%%%%%%%%%%%%%%%%%%%%%%%%%%%%%%%%%%%%%%%%%%%%%%%%%%%%%%%%

Let’s make the above argument more precise. In the following theorem we show how to construct a Boolean algebra out of the set of all  finite sequences of formulas of a given logic $L$.

\begin{theorem}\label{thm: Generic Example of term Boolean algebra}
	Given a logic $L$ formulated as a Hilbert system, there is a Boolean algebra associated to the proofs of $L$.
\end{theorem}

\begin{proof}
	Consider a logic $L$ formulated as a Hilbert system. 
	A proof of $\phi$ from a set of assumptions $\Gamma$ in the logic $L$ is a finite sequence of formulas, say $\langle \phi_1, \phi_2,\ldots,\phi_n \rangle$, such that each formula $\phi_i$ in the sequence is either an axiom instance of $L$ or is in $\Gamma$ or is obtained from previous formulas using rules, and in addition $\phi_n = \phi$. We consider only single-conclusion proofs here, i.e. every proof is a proof of only one formula. Let $\mathsf{Pr}$ be the set of all proofs from assumptions in $L$. 
	Given $\sigma \in \mathsf{Pr}$, we write  $\sigma_\phi^\Gamma$ if $\sigma$ is a proof of $\phi$ using the set of assumptions $\Gamma$. Note that the sequence $\langle \phi \rangle$ denotes a proof of $\phi$ from the assumption $\phi$ (i.e. a proof for the sequent $\phi \vdash_L \phi$).  For $\sigma_\phi^\Gamma \in \mathsf{Pr}$, let 
	%	\[
	%	\OCl{\sigma_\phi} := \{  \tau_\psi \in \mathsf{Pr} \mid (\exists \Gamma) \Gamma \vdash \phi \liff \psi  \}
	%	\]
	%	or ?
	\[
	\OCl{\sigma_\phi^\Gamma} := \{  \tau_\psi^\Delta \in \mathsf{Pr} \mid \vdash_L \phi \liff \psi  \}.
	\]

	Let $\OCl{\mathsf{Pr}} := \{ \OCl{\sigma_\phi^\Gamma} \mid \sigma_\phi^\Gamma \in \mathsf{Pr} \}$. Now we define the Boolean operators on $\OCl{\mathsf{Pr}}$ as follows:
	\begin{align*}
		0_\mathsf{Pr} &:= \OCl{\langle \bot \rangle},\\
		-_\mathsf{Pr} \OCl{\sigma_\phi^\Gamma} &:= \OCl{\langle \neg \phi \rangle}, \\
		\OCl{\sigma_\phi^\Gamma} \ttsum_\mathsf{Pr} \OCl{\tau_\psi^\Delta} &:= \OCl{\langle \phi \vee \psi \rangle}.
	\end{align*}
	
%	We start by defining $0_\mathsf{Pr} := \OCl{\langle \bot \rangle}$. Define $-_\mathsf{Pr} \OCl{\sigma_\phi} := \OCl{\langle \neg \phi \rangle}$ and $\OCl{\sigma_\phi} \ttsum_\mathsf{Pr} \OCl{\tau_\psi} := \OCl{\langle \phi \vee \psi \rangle}$. 
	%	We define $0$ as $\OCl{\langle \bot \rangle}$; $\OCl{\sigma_\phi^\Gamma} \ttsum \OCl{\tau_\psi^\Delta}$ as  $\OCl{\omega_{\phi \vee \psi}^{\Gamma \cup \Delta}}$; and $- \OCl{\sigma_\phi^\Gamma}$ as $\OCl{\tau_{\neg \phi}^{\neg \phi}}$, in which 
	%	%
%		\begin{align*}
%			\omega_{\phi \vee \psi}^{\Gamma \cup \Delta} &= \sigma_\phi \sharp \tau_\psi \sharp \langle \psi \limplies \phi \vee \psi, \phi \vee \psi \rangle \\
%		\tau_{\neg \phi}^{\neg \phi} &= \langle \neg \phi \rangle
%		\end{align*}
	%where $\sharp$ denotes the concatenation of sequences.
	Then, one can observe that $1_\mathsf{Pr}$ is $\OCl{\langle \top \rangle}$, and $\OCl{\sigma_\phi} \odot_\mathsf{Pr} \OCl{\tau_\psi}$ is  $\OCl{\langle \phi \wedge \psi \rangle}$.\footnote{Another intuitive way to define $\ttsum_\mathsf{Pr}$ is $\OCl{\sigma_\phi^\Gamma} \ttsum_\mathsf{Pr} \OCl{\tau_\psi^\Delta} := \OCl{\pi_{\phi \vee \psi}^{\Gamma \cup \Delta}}$, where $\pi_{\phi \vee \psi}^{\Gamma \cup \Delta} = \sigma_\phi^\Gamma \sharp \tau_\psi^\Delta \sharp \langle \psi \limplies \phi \vee \psi, \phi \vee \psi \rangle$ and $\sharp$ denotes the concatenation of sequences. Since $\OCl{\pi_{\phi \vee \psi}^{\Gamma \cup \Delta}} =  \OCl{\langle \phi \vee \psi \rangle}$, the two definitions are equivalent.}
	
	Now it is straightforward to show that $\Algebra_\mathsf{Pr} = (\OCl{\mathsf{Pr}}, 0_\mathsf{Pr}, -_\mathsf{Pr}, \ttsum_\mathsf{Pr})$ is a Boolean algebra. 	\qed
	
\end{proof}

In the above proof note that if $\psi$ is an axiom of $L$, then $\OCl{\langle \psi \rangle}$ plays a role similar to a proof constant, and if $\chi$ is a formula which is not a theorem of $L$, then $\OCl{\langle \chi \rangle}$ plays the same role as a proof variable which has been already replaced by a proof of $\chi$ (from the assumption $\chi$). The Boolean algebra $\Algebra_\mathsf{Pr}$ gives a \textit{generic} example of Boolean algebras on terms in the sense that for every $L$ (formulated as a Hilbert system) one can construct a Boolean algebra associated to the proofs of $L$. %$(\TermsBool, \ttsum, -, 0)$.

\begin{remark}
	We continue the proof of Theorem \ref{thm: Generic Example of term Boolean algebra} by defining the term operators $\tapp$ and $\tinspect$ on $\Algebra_\mathsf{Pr}$. To this end we assume that the logic $L$ satisfies the Lifting Lemma \ref{lem: Lifting}. Define $\tapp_\mathsf{Pr}$ as follows:
	\[
	\OCl{\sigma_\phi^\Gamma} \tapp_\mathsf{Pr} \OCl{\tau_\psi^\Delta} :=
	\begin{cases}
		\OCl{\omega_{\phi_2}^{\Gamma \cup \Delta}} & \mbox{ if $\phi = \phi_1 \limplies \phi_2$ and $\psi = \phi_1$} \\
		0_\mathsf{Pr} & \mbox{ Otherwise.}
	\end{cases} 
	\]
	In order to define $\tinspect_\mathsf{Pr}$ we use the Lifting Lemma. Suppose that $\Gamma \vdash \phi$ and $\sigma$ denotes this proof. By the Lifting Lemma we get 
	\[
	\vec{\justVarOne}: \Gamma \vdash t(\vec{\justVarOne}):\phi,
	\]
	for some term $t(\vec{\justVarOne}) \in \Terms$   (here by $\vec{\justVarOne}: \Gamma$ we mean $\justVarOne_1: \psi_1, \ldots, \justVarOne_n: \psi_n$ where $\Gamma = \{ \psi_1, \ldots, \psi_n \}$). Now one possible definition of $\tinspect_\mathsf{Pr}$ is as follows:
	\[
		\tinspect_\mathsf{Pr} \OCl{\sigma_\phi^\Gamma} := \OCl{\tau_{\jbox{t(\vec{\justVarOne})}\ \phi}^{\vec{\justVarOne}: \Gamma}},
	\]
	where $\tau$ denotes the proof of the sequent $\vec{\justVarOne}: \Gamma \vdash t(\vec{\justVarOne}):\phi$.
	For example, if $\sigma = \langle p, p \limplies p \vee q, q \rangle$, then 
	\[
	\tinspect_\mathsf{Pr} \OCl{\sigma} = \OCl{ \langle \justVarOne: p, \justConsThree:(p \limplies p \vee q), \justConsThree \tapp \justVarOne: q \rangle },
	\]
	where $\justVarOne \in \VTerms$ and $\justConsThree \in \CTerms$.

\end{remark}

Since we have a Boolean algebra on proof terms, we have an order $\leq$ on proof terms. For $s, t \in \TermsBool$:
\begin{center}
	$s \leq t$ iff $s \ttsum t \approx t$.
\end{center}
For example, in the algebra $\Algebra_\mathsf{Pr}$ we have
\begin{align*}
	\OCl{\sigma_\phi^\Gamma} \leq \OCl{\tau_\psi^\Delta} \quad & \text{ iff } \quad \OCl{\sigma_\phi^\Gamma} \ttsum_\mathsf{Pr} \OCl{\tau_\psi^\Delta} = \OCl{\tau_\psi^\Delta} \\
	& \text{ iff } \quad \OCl{\langle \phi \vee \psi \rangle} = \OCl{\tau_\psi^\Delta} \\
	& \text{ iff } \quad \vdash_L \phi \vee \psi \liff \psi \\
	& \text{ iff } \quad \vdash_L \phi \limplies \psi.
\end{align*}
%%%%%%%%%%%%%%%%%%%%%%%%%%%%%%%%%%%%%%%%%%%%%%%%%%%%%%%%%%

Next we present an axiomatic formulation for $\LPBool$. The set of axioms of $\LPBool_\emptyset$ is an extension of that of $\LP_\emptyset$ with  the following  axioms of the Boolean algebra for terms:
\begin{description}
	\item[B1.] $s \ttsum t \lequal t \ttsum s$ \hfil $s \odot t \lequal t \odot s$
	
	\item[B2.] $s\ttsum(t\ttsum u)\lequal(s\ttsum t)\ttsum u$ \hfil $s\odot(t\odot u)\lequal(s\odot t)\odot u$
	
	\item[B3.] $s\ttsum\Abot\lequal s$ \hfil $s\odot\Atop\lequal s$
	
	\item[B4.] $s\ttsum(-s)\lequal\Atop$ \hfil $s\odot(-s)\lequal\Abot$
	
	\item[B5.] $s\ttsum(t\odot u)\lequal(s\ttsum t)\odot(s\ttsum u)$ \hfil $s\odot(t\ttsum u)\lequal(s\odot t)\ttsum(s\odot u)$
\end{description}

and the following axioms for equality: 
\begin{description}
	\item[Eq1.] $t \lequal t$
	
	\item[Eq2.] $s \lequal t \wedge \phi[x/s] \limplies \phi[x/t]$, 
	
%	\item[EqTm. ] $$\frac{s \lequal t}{s:\phi \leftrightarrow t:\phi}$$
\end{description}
where $\phi[x/s]$ denotes the result of substitution of $s$ for $x$ in $\phi$.

The definition of constant specification and axiomatically appropriate constant specification is similar to that given in Section \ref{sec: LP-CS algebras} (just let $\JL$ be $\LPBool$). Given a constant specification $\CS$ for $\LPBool$, the axiomatic system $\LPBool_\CS$ is obtained by adding the formulas of $\CS$ as new axioms to the axiomatic system $\LPBool_\emptyset$. % and dropping the rule Int.
In addition, the rules of $\LPBool_\CS$ are $MP$ and the following rule:
\begin{description}
	\item[Int.] $$\frac{\phi}{\Atop:\phi}$$
\end{description}
The rule $Int$ says that $\Atop$ is a proof of every theorem.\footnote{Another possible interpretation of $\Atop$ is to consider it as a justification of everything, namely a \textit{universal justification} for every proposition. Regarding this interpretation, the rule $Int$ should be replaced by the rule $\frac{}{\Atop:\phi}$.} In this respect, $\Atop$ can be called a \textit{universal proof}. The rule $Int$ is enough to show that the internalization property holds in $\LPBool_\CS$ for arbitrary (not necessarily axiomatically appropriate) constant specification $\CS$.

\begin{lemma}[Internalization via the universal proof]
	Given an arbitrary constant specification $\CS$, if $\vdash_{\LPBool_\CS} \phi$, then  $\vdash_{\LPBool_\CS} \Atop:\phi$.
\end{lemma}
\begin{proof}
Follows immediately from the rule \textbf{Int}. \qed
\end{proof}

However, in contrast to the Internalization Lemma \ref{lem: Internalization},  the structure of the proof of theorems are not reflected by the universal proof $\Atop$. Thus, we next show that the standard Internalization Lemma also holds for $\LPBool$.

\begin{lemma}[Internalization] \label{lem: Internalization for LPBool}
	Suppose that $\CS$ is an axiomatically appropriate constant specification for $\LPBool$. If $\vdash_{\LPBool_\CS} \phi$, then there is a term $t \in \TermsBool$ such that $\vdash_{\LPBool_\CS} t:\phi$.
\end{lemma}
\begin{proof}
	The proof is similar to that of Lemma \ref{lem: Internalization}. The only new case is when $\phi$ is obtained by the rule $Int$. Suppose that $\phi = \Atop : \psi$ is obtained by the rule $Int$ from $\psi$. Then, using axiom $j4$ and $MP$, we get $\tinspect \Atop :  \Atop : \psi$. Thus, put $t := \tinspect \Atop$. \qed
\end{proof}

\begin{lemma}
	The following rule is admissible in $\LPBool_\CS$:
	
	 $$\frac{s \lequal t}{s:\phi \leftrightarrow t:\phi}\ EqTm$$
\end{lemma}
\begin{proof}
	Follows from axiom Eq2. \qed
\end{proof}

We immediately get the following result.

\begin{lemma}
	The following are theorems of $\LPBool_\CS$:
	\begin{enumerate}
		\item $s \ttsum t : \phi \liff t \ttsum s : \phi$ \hfil $s \odot t : \phi \liff t \odot s : \phi$
		
		\item $s\ttsum(t\ttsum u) : \phi \liff (s\ttsum t)\ttsum u : \phi$ \hfil $s\odot(t\odot u) : \phi \liff (s\odot t)\odot u :\phi $
		
		\item $s\ttsum \Abot : \phi \liff s : \phi$ \hfil $s\odot\Atop : \phi \liff s : \phi$
		
		\item  $s\ttsum(-s) : \phi \liff \Atop : \phi$ \hfil $s\odot(-s) : \phi \liff \Abot : \phi$	
		
		\item $s\ttsum(t\odot u) : \phi \liff (s\ttsum t)\odot(s\ttsum u) : \phi$ \hfil $s\odot(t\ttsum u) : \phi \liff (s\odot t)\ttsum(s\odot u) : \phi$	
		
	\end{enumerate}
\end{lemma}
\begin{proof}
	Follows from the axioms of  Boolean algebra, $B1$-$B5$, and the admissible rule  $EqTm$. \qed
\end{proof}

%%%%%%%%%%%%%%%%%%%%%%%%%%%%%%%%%%%%%%%%%%%%%%%%%%%%%%%%%%%%%%%%%%%%
\subsection{Full $\LPBool_\CS$ algebras}

Next we set semantics for the logic $\LPBool_\CS$. Unlike the full $\LP_\CS$ algebras that have been defined over the set of proof terms $\Terms$, here we define  algebras over certain extensions of a Boolean algebra. We begin with a Boolean algebra $(T, 0_\T, -_\T, \ttsum_\T)$, and then we extend it by the operators $\tapp_\T : T \times T \to T$ and $\tinspect_\T :T \to T$ to $\T$. The members of $T$ are considered as denotations of proof terms. Then, for every $\alpha \in T$ we have a function $\OBox{\alpha} : \FormulaeBool \to A$ in the algebra. Thus, the difference, as compared to full $\LP$ algebras in Section \ref{sec: LP-CS algebras}, is that now for every member of $\alpha \in T$ we have a function $\OBox{\alpha}$ in the algebra. As opposed to full $\LP$ algebras in which the number of functions $\OBox{t}$, for $t \in \Terms$, were countably infinite, since $T$ may be finite, the number of functions $\OBox{\alpha}$ may be finite as well.

\begin{definition}\label{def: pre-LPB-CS algebra}
	Let $\T = (T, 0_\T, -_\T, \ttsum_\T,  \tapp_\T,  \tinspect_\T)$, where $(T, 0_\T, -_\T, \ttsum_\T)$ is a Boolean algebra, and $\tapp_\T$ and $\tinspect_\T$ are two operators on $T$ such that 
	$$
	\tapp_\T : T \times T \to T \quad \text{ and } \quad \tinspect_\T :T \to T.
	$$
	Given a constant specification $\CS$ for $\LPBool$, a pre $\LPBool_\CS$ algebra is a tuple $\Algebra = (A, \AbotElement, \Anot, \Ajoin, \OBox{\alpha}, I)_{\alpha \in T}$ such that $(A, \AbotElement, \Anot, \Ajoin)$ is a Boolean algebra, $I: \CTerms \cup \{ \Abot \} \to T$ is an interpretation such that 
	$$
	I(\Abot)=0_{\T},$$
	and operators $\OBox{\alpha} : \FormulaeBool \to A$, for $\alpha \in T$,  satisfy the following conditions. For all $\AElementOne , \AElementTwo  \in A$, and all $\alpha, \beta \in T$:
	\begin{description}
		%\item[WO-B.] $\OBox{\alpha} (\AElementOne )$ is well-ordered. 
		
		\item[Al-Appl-$\LPBool_\CS$.] $\OBox{\alpha}(\phi \limplies \psi) \Ameet \OBox{\beta}(\phi) \leq \OBox{\alpha \tapp_{\T} \beta}(\psi)$,
		
		\item[Al-Sum-$\LPBool_\CS$.] $\OBox{\alpha}(\phi) \Ajoin \OBox{\beta}(\phi) \leq \OBox{\alpha \ttsum_{\T}  \beta}(\phi)$,
		
		%\item[Al-jT-B.] $\OBox{\alpha}(\phi) \leq \BIExt(\phi)$,
		
		%		\item[Al-j4-$\LPBool_\CS$.] $\OBox{\alpha}(\phi) \leq \OBox{\tinspect_{T[\VTerms]} \alpha} (\OBox{\alpha}(\phi))$,
		
		\item[Al-$\Atop$--$\LPBool_\CS$] $\OBox{1_\T}(\phi) = \AtopElement$, where $\vdash_{\LPBool_\CS} \phi$,
		
		\item[Al-CS-$\LPBool_\CS$] $\OBox{I(\justConsThree)}(\phi) = \AtopElement$, where $\justConsThree:\phi \in \CS$.
		%		\item[Al-CS] $\OBox{c}(\AtopElement) = \AtopElement$ for $\justConsThree:\phi \in \CS$.
		
		%\item[Eq.] if $\alpha = \beta$, then $\forall \phi \in \FormulaeBool (\OBox{\alpha} (\phi) = \OBox{\beta} (\phi))$.  
	\end{description}
	
	%	Let $\nabla$ be a set of certain \textit{distinguished elements} of $A$.
	%	The class of all $\LPBool_\emptyset$ algebras with singleton $\nabla$ is denoted by $\Algebra_{\LPBool_\emptyset}^N$.
\end{definition}

\begin{definition}\label{def: t^I_v Section 5}
Given $\T = (T, 0_\T, -_\T, \ttsum_\T,  \tapp_\T,  \tinspect_\T)$ and a function $v: \VTerms \to T$, for every $t \in \TermsBool$ define $t^{v}_{I}$ as follows:
\begin{description}
	\item[1.] $\justVarOne^{v}_{I}:=v(\justVarOne)$ for $\justVarOne \in\VTerms$ and $\justConsThree^{v}_{I}:=I(\justConsThree)$ for $\justConsThree \in \CTerms \cup \{ \Abot \}$,
	
	\item[2.] $(s \ast t)^{v}_{I}:= (s^{v}_{I}) \ast_{\T} (t^{v}_{I})$ where $\ast\in\{+,\cdot\}$, 
	
	\item[3.] $(\ast  s)^{v}_{I}:= \ast_{\T} (s^{v}_{I})$ where $\ast\in\{-,!\}$.
	
\end{description}
Note that $t^{v}_{I} \in T$, for every $t \in \TermsBool$.
\end{definition}

\begin{definition}
	A valuation is a function $\Val : \Prop \to A$. Given a function $v: \VTerms \to T$, the assignment $\BIValExt : \FormulaeBool \to A$ on $\Algebra$ is defined as follows:

	\begin{align*}
	\BIValExt(p) &= \Val(p)\\
	\BIValExt(\bot) &= \AbotElement\\
	\BIValExt(\neg\phi) &= \Anot\BIValExt(\phi)\\
	\BIValExt(\phi \vee \psi) &= \BIValExt(\phi) \Ajoin \BIValExt(\psi) \\
	\BIValExt(t: \phi) &= \OBox{t^{v}_{I}}(\phi) \\
	%	\ValExt(s \lequal t) = \left{ 
	\BIValExt(s \lequal t) & = 
	\begin{cases}
	\AtopElement & \text{ if } s^{v}_{I} = t^{v}_{I}\\
	\AbotElement & \text{ if } s^{v}_{I} \neq t^{v}_{I}.
	\end{cases} 
	%	\ValExt(s \lequal t) = \AtopElement &\text{ iff } \ValExt(s: \phi) = \ValExt(t: \phi), \text{for all } \phi.
	\end{align*}
	
\end{definition}

\begin{definition}\label{def: full LPB-CS algebra}
	A full $\LPBool_\CS$ algebra $\Algebra = (A, \AbotElement, \Anot, \Ajoin, \OBox{\alpha}, I)_{\alpha \in T}$ is a pre $\LPBool_\CS$ algebra that satisfies the following condition. For all $\phi \in \FormulaeBool$, all $t \in \TermsBool$, all $\alpha \in T$, all functions $v: \VTerms \to T$, and  all valuation $\Val: \Prop \to A$:
	\begin{description}
		\setlength\itemsep{0.1cm}
		
		\item[Al-j4-$\LPBool_\CS$.] $\OBox{\alpha}(\phi) \leq \OBox{\tinspect_{\T} \alpha} (t:\phi)$,
		
		\item[Al-jT-$\LPBool_\CS$.] $\OBox{\alpha}(\phi) \leq \BIValExt(\phi).$
		
	\end{description}
	The class of all full $\LPBool_\CS$ algebras with singleton $\nabla  = \{\AtopElement\}$ is denoted by $\Algebra_{\LPBool_\CS}^{\sf full}$.
\end{definition}

\begin{definition}[Validity]
	Let $\Algebra = (A, \AbotElement, \Anot, \Ajoin, \OBox{\alpha}, I)_{\alpha \in T}$ be an $\LPBool_\emptyset$ algebra (with $\nabla  = \{\AtopElement\}$) and $\phi \in \FormulaeBool$.
	\begin{itemize}
		\setlength\itemsep{0.1cm}
		
		\item A formula $\phi$ is true in the algebra $\Algebra$ under the valuation $\theta$, denoted by $\Algebra \models_\theta \phi$, if $\BIValExt(\phi) = \AtopElement$, for every function $v: \VTerms \to T$. 
		
		\item A formula $\phi$ is true in the algebra $\Algebra$, denoted by $\Algebra \models \phi$, if $\Algebra \models_\theta \phi$ for every valuation $\Val$ on $\Algebra$. 
		
		\item A formula $\phi$ is valid in the class $\Algebra_{\LPBool_\CS}^{\sf full}$, denoted by $\Algebra_{\LPBool_\CS}^{\sf full} \models \phi$, if $\Algebra \models \phi$ for every algebra $\Algebra \in \Algebra_{\LPBool_\CS}^{\sf full}$.
	\end{itemize}
\end{definition}

%%%%%%%%%%%%%%%%%%%%%%%%%%%%%%%%%%%%%%%%%%%%%%%%%%%
\subsection{Completeness of $\LPBool_\CS$}

In order to prove completeness of $\LPBool_\CS$ we construct two Tarski-Lindenbaum algebras; one out of $\FormulaeBool$ and the other out of $\TermsBool$. 

\begin{definition}\label{def: Tarski-Lindenbaum algebra for LPBool_emptyset}
	Given $\phi \in \FormulaeBool$ and $t \in \TermsBool$, let 
	\[
	\OCl{\phi} := \{ \psi \in \FormulaeBool \mid \vdash_{\LPBool_\CS} \phi \liff\psi \},
	\]
	\[
	\OCl{t} := \{ s \in \TermsBool \mid \vdash_{\LPBool_\CS} s \lequal t \}.
	\]
	First define $\T = (\OCl{\TermsBool}, 0_\T, -_\T, \ttsum_\T, \tapp_\T,  \tinspect_\T)$ as follows:
	\begin{align*}
		\OCl{\TermsBool} &:= \{ \OCl{t} \mid t \in \TermsBool \}, \\
			0_\T  &:= \OCl{\Abot}, \\
				-_\T \OCl{t}  &:= \OCl{- t}, \\
				  \OCl{s} \ttsum_\T \OCl{t} &:= \OCl{s \ttsum t}, \\
					\OCl{s} \tapp_\T \OCl{t} &:= \OCl{s \tapp t}, \\
						\tinspect_\T \OCl{t} &:= \OCl{\tinspect t}.
	\end{align*}
	The Tarski-Lindenbaum algebra for $\LPBool_\CS$ is defined as follows:	
	\[
	\Algebra_{\LPBool_\CS}^f := ( \OCl{\FormulaeBool}, \AbotElement, \Anot, \Ajoin, \OBox{\OCl{t}}, I)_{\OCl{t} \in \OCl{\TermsBool}},
	\]
	where
	\begin{align*}
		\OCl{\FormulaeBool} &:= \{ \OCl{\phi} \mid \phi \in \FormulaeBool \}, \\
		\OBox{\OCl{t}} (\phi) &:= \OCl{t:\phi}, \\
		I(\justConsThree) &:= \OCl{\justConsThree}, \qquad \mbox{ for } \justConsThree \in \CTerms \cup \{ \Abot\},
	\end{align*} 
%	\[
%	\OCl{\FormulaeBool} := \{ \OCl{\phi} \mid \phi \in \FormulaeBool \},
%	\]
%	\[
%	\OBox{\OCl{t}} (\phi) := \OCl{t:\phi},
%	\]
%	%
%	\[
%	I(\justConsThree) := \OCl{\justConsThree}, \qquad \mbox{ for $\justConsThree \in \CTerms \cup \{ \Abot\},$ }
%	\]
	and $\AbotElement$, $\Anot$, and $\Ajoin$ are defined similar to Definition \ref{def: Tarski-Lindenbaum algebra for HLP}. 
	Let $\nabla =  \{ \OCl{\top} \}$. %$\nabla =  \{ \OCl{\phi} \mid \vdash \phi \}$.
\end{definition}

\begin{lemma}
	The Tarski-Lindenbaum algebra $\Algebra_{\LPBool_\CS}^f$ is a pre $\LPBool_\CS$ algebra.
\end{lemma}
\begin{proof}
	We only check $Al$-$Appl$-$\LPBool_\CS$ and $Al$-$\Atop$-$\LPBool_\CS$. The proof of other cases is simple.
	\begin{itemize}
		\item Condition $Al$-$Appl$-$\LPBool_\CS$.
		\begin{align*}
			& \OBox{\OCl{s}}(\phi \limplies \psi) \Ameet \OBox{\OCl{t}}(\phi) \\ &=
			\OCl{\jbox{s} (\phi \limplies \psi)} \Ameet \OCl{\jbox{t} \phi} \\  &=
			\OCl{\jbox{s} (\phi \limplies \psi) \wedge \jbox{t} \phi} \\  &\leq
			\OCl{\jbox{s \tapp t} \psi} \\ &=  
			 \OBox{\OCl{s \tapp t}} (\psi) \\ &=
			 \OBox{\OCl{s} \tapp_\T \OCl{t}} (\psi).
		\end{align*}
	
	\item Condition $Al$-$\Atop$-$\LPBool_\CS$. Suppose that $\vdash_{\LPBool_\CS} \phi$. Then
	\[
	\OBox{\OCl{\Atop}} (\phi) =   \OCl{\Atop : \phi}  = \OCl{\top}. 
	\]
	\qed
	\end{itemize}
%	Let us check the $Eq$ condition. Suppose that $\OCl{s} = \OCl{t}$. Then, $\vdash s \lequal t$. Using $EqTm$ we get $\vdash \jbox{s} \phi \liff \jbox{t} \phi$, for all $\phi \in \FormulaeBool$. Therefore, $\OCl{\jbox{s} \phi} = \OCl{\jbox{t} \phi}$, for all $\phi \in \FormulaeBool$, and hence $\OBox{\OCl{s}} (\OCl{\phi}) = \OBox{\OCl{t}} (\OCl{\phi})$, for all $\OCl{\phi} \in \OCl{\FormulaeBool}$.
\end{proof}

\begin{theorem}[Soundness and completeness]\label{thm: completeness LPBool over arbitrary Boolean algebra}
	$\vdash_{\LPBool_\CS} \phi$ iff $\Algebra_{\LPBool_\CS}^{\sf full} \models \phi$.
\end{theorem}
\begin{proof}
	Soundness is straightforward. For completeness, define $\Val$ and $v$ as follows:
	\begin{align*}
	 	\OVal{p} & := \OCl{p}, \qquad \text{ for } p \in \Prop,\\
	 		v(\justVarOne) & := \OCl{\justVarOne}, \qquad \text{ for } \justVarOne \in \VTerms,
	\end{align*}
and let $\Algebra_{\LPBool_\CS}^f := ( \OCl{\FormulaeBool}, \AbotElement, \Anot, \Ajoin, \OBox{\OCl{t}}, I)_{\OCl{t} \in \OCl{\TermsBool}}$ be the Tarski-Lindenbaum algebra for $\LPBool_\CS$. Then, by induction on the complexity of $t$, one can show that	
\begin{equation}\label{eq: t^{v}_{I} = \OCl{t} completeness of LPBool_emptyset}
	t^{v}_{I} = \OCl{t}.
\end{equation}
Then we show the Truth Lemma:
\[
\BIValExt(\phi) = \OCl{\phi}.
\]
The proof is by induction on the complexity of $\phi$. We only check two cases:
\begin{itemize}
	\setlength\itemsep{0.1cm}
	
	\item Case $\phi = t:\psi$: 
	
	$\BIValExt(t: \phi) =  \OBox{t^{v}_{I}}(\phi) =  \OBox{\OCl{t}}(\phi) = \OCl{t:\phi}$.
	
	\item Case $\phi = s \lequal t$:
	
	$\BIValExt(s \lequal t) = \OCl{\top}$ iff $s^{v}_{I} = t^{v}_{I}$ iff $\OCl{s} = \OCl{t}$ iff $\vdash_{\LPBool_\CS} s \lequal t$ iff $\OCl{s \lequal t} = \OCl{\top}$. \\
	Thus, $\BIValExt(s \lequal t) = \OCl{s \lequal t}$. 
\end{itemize}
Again completeness follows from the Truth Lemma. \qed
\end{proof}

%%%%%%%%%%%%%%%%%%%%%%%%%%%%%%%%%%%%%%%%%%%%%%%%%%%%%%%%%%%5
\subsection{Bi-representation theorem for full $\LPBool_\CS$ algebras}
\label{sec: Bi-representation theorem LPB-emptyset}

In this section, we present a counterpart of the Stone's representation theorem in this setting by means of certain isomorphisms for full $\LPBool_\CS$-algebras.

\begin{definition}
	Let $\Algebra = (A, \AbotElement, \Anot, \Ajoin, \OBox{\alpha}, I)_{\alpha \in T}$ and $\Algebra' = (A', \AbotElement', \Anot', \Ajoin', \Box'_{\alpha}, I')_{\alpha \in T'}$ be two full $\LPBool_\CS$ algebras over $\T = (T, 0_\T, -_\T, \ttsum_\T, \tapp_\T, \tinspect_\T)$ and $\T' = (T', 0_{\T'}, -_{\T'}, \ttsum_{\T'}, \tapp_{\T'}, \tinspect_{\T'})$ respectively. A bi-isomorphism  between $\Algebra$ and $\Algebra'$ is a pair of Boolean isomorphisms $(f,g)$ such that $f: (A, \AbotElement, \Anot, \Ajoin) \to (A', \AbotElement', \Anot', \Ajoin')$ and  $g: (T, 0_\T, -_\T, \ttsum_\T) \to (T', 0_{\T'}, -_{\T'}, \ttsum_{\T'})$ satisfying the following conditions:
	\begin{align*}
		f(\OBox{\alpha}(\phi)) &= \Box'_{g(\alpha)}(\phi), \qquad \mbox{ for $\alpha \in T$ and $\phi \in \FormulaeBool$, } 
		\\
		g(I(\justConsThree)) &= I'(\justConsThree), \qquad\mbox{ for $\justConsThree \in \CTerms \cup \{ \Abot\}$, } 
		\\
		g(\alpha \tapp_\T \beta) &= g(\alpha) \tapp_{\T'} g(\beta) ,  \qquad\mbox{ for $\alpha, \beta \in T$, } 
		\\
		g(\tinspect_\T \alpha) &=	\tinspect_{\T'} g(\alpha), \qquad\mbox{ for $\alpha \in T$}.
	\end{align*}
\end{definition}

\begin{definition}
	Given two set algebras $(A, \emptyset, \setminus, \cup)$ and $(B, \emptyset, \setminus, \cup)$, a full $\LPBool_\CS$ set algebra is a structure 
	$$\Algebra = (A, \emptyset, \setminus, \cup, \OBox{\alpha}, I)_{\alpha \in B}$$ 
	 over $\mathcal{B} = (B, \emptyset, \setminus, \cup, \tapp_\mathcal{B}, \tinspect_\mathcal{B})$ satisfying the conditions of Definitions \ref{def: pre-LPB-CS algebra} and \ref{def: full LPB-CS algebra}. 
\end{definition}

In contrast with the standard representation theorems we construct two isomorphisms for a given full $\LPBool_\CS$ algebra $\Algebra$ over $\T$: One for the algebra $\Algebra$ and one for the algebra $\T$. This leads to a bi-isomorphism of full $\LPBool_\CS$ algebras.

\begin{theorem}[Bi-Representation Theorem]\label{thm: Representation Theorem LPBool}
	Every full $\LPBool_\CS$ algebra is bi-isomorphic to a full $\LPBool_\CS$ set algebra.
	%For every $\LPBool_\CS$ algebra $\Algebra = (A, \AbotElement, \Anot, \Ajoin, \OBox{\alpha}, I)_{\alpha \in T}$ over $\T = (T, 0_\T, -_\T, \ttsum_\T, \tapp_\T, \tinspect_\T)$, there exists a power set $\LPBool_\CS$ algebra $\mathcal{B} = (B,  \Ajoin', \Anot', \AbotElement', \OBox{\alpha}, I')_{\alpha \in \mathcal{P}_\T}$ which is bi-isomorphic to $\Algebra$.
\end{theorem}

\begin{proof}
%	Follows from the Stone representation theorem and Lemma \ref{lem: Bi-Representation Bool}. 
	Suppose that $\Algebra = (A, \AbotElement, \Anot, \Ajoin, \OBox{\alpha}, I)_{\alpha \in T}$ is a full $\LPBool_\CS$ algebra over $\T = (T, 0_\T, -_\T, \ttsum_\T, \tapp_\T, \tinspect_\T)$. Considering the Boolean algebra $(T, 0_\T, -_\T, \ttsum_\T)$, by the Stone representation theorem, there exists an isomorphism $g$ from $(T, 0_\T, -_\T, \ttsum_\T)$ to a set algebra $(P, \emptyset, \setminus, \cup)$.  Next, we shall extend $g$ to an isomorphism from $\T = (T, 0_\T, -_\T, \ttsum_\T, \tapp_\T, \tinspect_\T)$ to an extension of $(P, \emptyset, \setminus, \cup)$. 
	
	Let  $\alpha', \beta' \in P$. There are $\alpha, \beta \in T$ such that $g(\alpha) = \alpha'$  and $g(\beta) = \beta'$. Now define term operators $\tapp_\mathcal{P}$ and $\tinspect_\mathcal{P}$ on $P$ as follows:
	\begin{align}
		\alpha' \tapp_\mathcal{P} \beta' &:= g(\alpha \tapp_\T \beta), \\
		\tinspect_\mathcal{P} \alpha' &:=  g(\tinspect_\T \alpha).
	\end{align}
	Since $\tapp_\T$ and $\tinspect_\T$ are well-defined, it follows that $\tapp_\mathcal{P}$ and $\tinspect_\mathcal{P}$ are well-defined. Now it is obvious that $g$ is an isomorphism from $\T = (T, 0_\T, -_\T, \ttsum_\T, \tapp_\T, \tinspect_\T)$ to  $(P, \emptyset, \setminus, \cup, \tapp_\mathcal{P}, \tinspect_\mathcal{P})$. Let $\mathcal{P} = (P, \emptyset, \setminus, \cup, \tapp_\mathcal{P}, \tinspect_\mathcal{P})$.

	Next, by the Stone representation theorem, the Boolean algebra $(A, \AbotElement, \Anot, \Ajoin)$ is isomorphic to a set algebra $(B, \emptyset, \setminus, \cup)$ via an isomorphism $f$. Define $\OBox{\alpha'}^\prime$ (for $\alpha' \in P$) and the interpretation $I'$ as follows: 
	\begin{equation}\label{eq: induced function via bi-isomorphism Bool}
		\OBox{g(\alpha)}^\prime (\phi) :=  f(\OBox{\alpha}(\phi)), \qquad \text{ for } \alpha \in T.
	\end{equation}
	\begin{equation}\label{eq: induced function via bi-isomorphism Bool-interpretation}
		I'(\justConsThree) := g(I(\justConsThree)), \qquad\mbox{ for $\justConsThree \in \CTerms \cup \{ \Abot\}$. }
	\end{equation}
	
	Let $\mathcal{B} = (B, \emptyset, \setminus, \cup, \OBox{\alpha'}^\prime, I')_{\alpha' \in P}$. Then, we show that $\mathcal{B}$ is a full $\LPBool_\CS$ algebra. %The proof is similar to that of Lemma \ref{lem:pre-stone LP-empty}.
	To this end we only check conditions $Al$-$Appl$-$\LPBool_\CS$, $Al$-$\CS$-$\LPBool_\CS$, $Al$-$j4$-$\LPBool_\CS$ and $Al$-$jT$-$\LPBool_\CS$ from Definitions \ref{def: pre-LPB-CS algebra} and \ref{def: full LPB-CS algebra} (other cases can be verified similarly).
	
	%	For \textbf{EqTm}, suppose $g(s) = g(t)$. Then $s = t$, and
	%	%
	%	\[
	%	\OBox{g(s)}^\mathcal{B} (f(\AElementOne )) = \{ f(\AElementTwo ) \mid \AElementTwo  \in \OBox{s}^\Algebra(\AElementOne ) \} = \{ f(\AElementTwo ) \mid \AElementTwo  \in \OBox{t}^\Algebra(\AElementOne ) \} = \OBox{g(t)}^\mathcal{B} (f(\AElementOne )).
	%	\]
	
	For $Al$-$Appl$-$\LPBool_\CS$, suppose that $\alpha, \beta \in T$:
	\begin{eqnarray*}
		& \OBox{g(\alpha)}^\prime(\phi \limplies \psi) \Ameet \OBox{g(\beta)}^\prime(\phi)=\\		
		& f(\OBox{\alpha}(\phi \limplies \psi)) \Ameet f(\OBox{\beta}(\phi)) =
		\\		
		& f(\OBox{\alpha}(\phi \limplies \psi) \Ameet \OBox{\beta}(\phi)) \leq \\ 
		& f(\OBox{\alpha \tapp_{\T} \beta}(\psi))= \OBox{g(\alpha) \tapp_{\mathcal{P}_\T} g(\beta)}^\prime(\psi).
	\end{eqnarray*}
	
	For $Al$-$\CS$-$\LPBool_\CS$, suppose $\justConsThree:\phi \in \CS$, then we have
	\[
	\OBox{I'(\justConsThree)}^\prime (\phi) = \OBox{g(I(\justConsThree))}^\prime (\phi) = f(\OBox{I(\justConsThree)}(\phi)) = f(\AtopElement) = B.
	\]
	
	For $Al$-$j4$-$\LPBool_\CS$, suppose that $t \in \TermsBool$ is an arbitrary term. Then, for $\alpha \in T$, we have
	\begin{align*}
	 \OBox{g(\alpha)}^\prime (\phi) = f( \OBox{\alpha}(\phi))   \leq  f(\OBox{\tinspect_{\T} \alpha} (t:\phi)) = \OBox{g(\tinspect_{\T} \alpha)}^\prime (t:\phi)  = \OBox{\tinspect_{\mathcal{P}} g(\alpha)}^\prime (t:\phi).
	\end{align*}
	%	\[
	%	\OBox{\tilde{I'}(t)}^\prime (\phi) = \OBox{g(\alpha)}^\prime (\phi) = f( \OBox{\alpha}(\phi)) = f( \OBox{\tilde{I}(t)}(\phi)) \leq f(\OBox{\tinspect_{T[\VTerms]} \tilde{I}(t)} (t:\phi)) = f(\OBox{\tinspect_{T[\VTerms]} \alpha} (t:\phi)) = \OBox{g(\tinspect_{T[\VTerms]} \alpha)}^\prime (t:\phi)
	%	\]

	For $Al$-$jT$-$\LPBool_\CS$, we need some preliminary results. Given a function $v: \VTerms \to T$, define the function $v': \VTerms \to P$ by
	\begin{equation}\label{eq: Bi-representation def of function v'}
		v'(\justVarOne) := g(v(\justVarOne)).
	\end{equation}
	It is not difficult to show that for every $t \in \TermsBool$ we have
	\[
	t_{I'}^{v'} = g(t_I^v).
	\]
	The proof involves a routine induction on the complexity of $t$. The base cases, where $t \in \VTerms$ or $t \in \CTerms \cup \{ \Abot \}$, follows from \eqref{eq: induced function via bi-isomorphism Bool-interpretation} and \eqref{eq: Bi-representation def of function v'}. The proof for the induction steps is straightforward.
	
	Now suppose that $\Val' : \Prop \to B$ is  an arbitrary valuation on $\mathcal{B}$. Then, $\Val = f^{-1} \circ \Val'$ would be a valuation on $\Algebra$. Then, by induction on the complexity of the formula $\phi$, one can show that
	\[
	f(\BIValExt (\phi)) = \tilde{\theta'}_{v'} (\phi).
	\]
	The base case follows from $\Val = f^{-1} \circ \Val'$. For the induction step, we only show the case where $\phi = t:\psi$. We have
	\[
	f(\BIValExt (t:\psi)) = f(\OBox{ t_I^v } (\psi)) = \OBox{ g(t_I^v) }' (\psi) = \OBox{ t_{I'}^{v'} }' (\psi) = \tilde{\theta'}_{v'} (t:\psi).
	\]
	Thus, $f \circ \BIValExt = \tilde{\theta'}_{v'}$. Finally, we verify the condition $Al$-$jT$-$\LPBool_\CS$. We have
	\[
	\OBox{g(\alpha)}^\prime (\phi) = f( \OBox{\alpha}(\phi)) \leq f(\BIValExt (\phi)) = \tilde{\theta'}_{v'} (\phi).
	\]
	The above inequality follows from the fact that $f$ is an isomorphism  and $\OBox{\alpha}(\phi) \leq \BIValExt (\phi)$.
	
	It is obvious that $\mathcal{B}$ is a set algebra, and hence it is a full $\LPBool_\CS$ set algebra.	\qed
\end{proof}

\begin{theorem}
	$\LPBool_\CS$ is complete with respect to the class of all full $\LPBool_\CS$ set algebras.
\end{theorem}
\begin{proof}
	Follows from the Bi-representation Theorem \ref{thm: Representation Theorem LPBool} and the Completeness Theorem \ref{thm: completeness LPBool over arbitrary Boolean algebra}. \qed
\end{proof}
%%%%%%%%%%%%%%%%%%%%%%%%%%

%%%%%%%%%%%%%%%%%%%%%%%%%%%%%%%%%%%%%%%%%%%%%%%%%%%%%%%%%%%%%%%%%%%%%%%%%

\section{$\LP_\CS$ algebras over arbitrary polynomial Boolean algebras}
\label{sec: LP-CS algebras over arbitrary polynomial Boolean algebras}

In this section, similar to Section \ref{sec: LP-CS algebras over arbitrary Boolean algebras}, we assume that proof terms constitute a Boolean algebra. We give another algebraic semantics for the logic $\LPBool$. In contrast to the results of Section \ref{sec: LP-CS algebras over arbitrary Boolean algebras}, in this section we use a Boolean algebra on polynomials. More precisely, while in Section \ref{sec: LP-CS algebras over arbitrary Boolean algebras} the truth of a formula in an algebra is defined directly using functions $v:\VTerms\to T$, where $T$ is a set of terms, in this section we make use of polynomials with coefficients in $T$.

%A constant specification $\CS$ for $\LPBool$ is a set of formulas of
%the form $\justConsThree:\phi $, where $c$ is a proof constant and $\phi $ is an axiom instance of $\LPBool$. Given a constant specification $\CS$ for $\LPBool$, the axiomatic system $\LPBool_\CS$ is obtained by adding the formulas of $\CS$ as new axioms to the axiomatic system $\LPBool_\emptyset$ of Section \ref{sec: LP-emptyset algebras over arbitrary Boolean algebras}.% and dropping the rule Int.

Let $\T = (T, 0_\T, -_\T, \ttsum_\T,  \tapp_\T,  \tinspect_\T)$, where $(T, 0_\T, -_\T, \ttsum_\T)$ is a Boolean algebra, and $\tapp_\T : T \times T \to T$ and $\tinspect_\T :T \to T$.\ Let $T[\VTerms]$ denote the set of polynomials with variables in $\VTerms$ and coefficients in $T$.\  Two polynomials $f,g\in T[\VTerms]$ are equal, denoted by $f \stackrel{p}{=} g$, if and only if for every $v:\VTerms\to T$, $v(f)=v(g)$, where by $v(f)$ we mean replacing every occurrence of every variable in $f$ by its image under $v$.  It follows from the equality defined above that the Boolean structure on $T$ induces a Boolean structure on  $T[\VTerms]$. Thus, we obtain a Boolean algebra $(T[\VTerms], 0_{\T}, -_{\T[\VTerms]},\ttsum_{\T[\VTerms]})$. Next we add the operators 
$$\tapp_{\T[\VTerms]} : T[\VTerms] \times T[\VTerms] \to T[\VTerms] \quad \text{ and } \quad \tinspect_{\T[\VTerms]} :T[\VTerms] \to T[\VTerms]$$
to this Boolean algebra in such a way that $\tapp_{\T[\VTerms]}$ and $\tinspect_{\T[\VTerms]}$ are extensions of $\tapp_{\T}$ and $\tinspect_{\T}$ respectively. For example, one can define $\tapp_{\T[\VTerms]}$ and $\tinspect_{\T[\VTerms]}$ as follows:
\begin{equation}\label{eq: definition of application in T[Var]}
	\alpha \tapp_{\T[\VTerms]} \beta :=  
	\begin{cases}
		\alpha \tapp_{\T} \beta, & \text{if } \alpha,\beta \in T, \\
		\alpha, & \text{Otherwise,}
	\end{cases}
\end{equation}
and
\begin{equation}\label{eq: definition of bang in T[Var]}
	\tinspect_{\T[\VTerms]} \alpha := 
\begin{cases}
	\tinspect_\T \alpha, & \text{if } \alpha \in T, \\
	\alpha, & \text{Otherwise.}
\end{cases}
\end{equation}
By $\T[\VTerms]$ we denote the structure
$$ (T[\VTerms], 0_{\T}, -_{\T[\VTerms]}, \ttsum_{\T[\VTerms]},  \tapp_{\T[\VTerms]},  \tinspect_{\T[\VTerms]}),$$
where $\tapp_{\T[\VTerms]}$ and $\tinspect_{\T[\VTerms]}$ are arbitrary extensions of $\tapp_{\T}$ and $\tinspect_{\T}$ respectively. From now on when we consider $\T[\VTerms]$, for some $\T$, we assume that  $\tapp_{\T[\VTerms]}$ and $\tinspect_{\T[\VTerms]}$ are extensions of $\tapp_{\T}$ and $\tinspect_{\T}$ respectively. 

%\todo{$\tapp_{\T[\VTerms]}$ should be defined in a way that extends $\tapp_\T$. }

\begin{definition}\label{def: LPB-CS algebra}
	Let $\T = (T, 0_\T, -_\T, \ttsum_\T,  \tapp_\T,  \tinspect_\T)$, where $(T, 0_\T, -_\T, \ttsum_\T)$ is a Boolean algebra, and $\tapp_\T : T \times T \to T$ and $\tinspect_\T :T \to T$, and let $\T[\VTerms] = (T[\VTerms], 0_{\T}, -_{\T[\VTerms]}, \ttsum_{\T[\VTerms]},  \tapp_{\T[\VTerms]},  \tinspect_{\T[\VTerms]})$. Given a constant specification $\CS$ for $\LPBool$, a pre polynomial $\LPBool_\CS$ algebra 
	$$\Algebra = (A, \AbotElement, \Anot, \Ajoin, \OBox{\alpha}, I)_{\alpha \in T[\VTerms]}$$
	is defined similar to pre $\LPBool_\CS$ algebra (Definition \ref{def: pre-LPB-CS algebra}), where  $I: \CTerms \cup \{ \Abot \} \to T$ is an interpretation such that $I(\Abot)=0_{\T}$, and operators $\OBox{\alpha} : \FormulaeBool \to A$, for $\alpha \in T[\VTerms]$,  satisfy the following conditions. For all $\AElementOne , \AElementTwo  \in A$, and all $\alpha, \beta \in T[\VTerms]$:
	\begin{description}
		%\item[WO-B.] $\OBox{\alpha} (\AElementOne )$ is well-ordered. 
		
		\item[Al-Appl-$\LPBool_\CS$.] $\OBox{\alpha}(\phi \limplies \psi) \Ameet \OBox{\beta}(\phi) \leq \OBox{\alpha \tapp_{T[\VTerms]} \beta}(\psi)$,
		
		\item[Al-Sum-$\LPBool_\CS$.] $\OBox{\alpha}(\phi) \Ajoin \OBox{\beta}(\phi) \leq \OBox{\alpha \ttsum_{T[\VTerms]}  \beta}(\phi)$,
		
		%\item[Al-jT-B.] $\OBox{\alpha}(\phi) \leq \BIExt(\phi)$,
		
%		\item[Al-j4-$\LPBool_\CS$.] $\OBox{\alpha}(\phi) \leq \OBox{\tinspect_{T[\VTerms]} \alpha} (\OBox{\alpha}(\phi))$,
		
		\item[Al-$\Atop$--$\LPBool_\CS$.] $\OBox{1_\T}(\phi) = \AtopElement$, where $\vdash_{\LPBool_\CS} \phi$,
		
		\item[Al-CS-$\LPBool_\CS$.] $\OBox{I(\justConsThree)}(\phi) = \AtopElement$, where $\justConsThree:\phi \in \CS$.
		%		\item[Al-CS] $\OBox{c}(\AtopElement) = \AtopElement$ for $\justConsThree:\phi \in \CS$.
		
		%\item[Eq.] if $\alpha = \beta$, then $\forall \phi \in \FormulaeBool (\OBox{\alpha} (\phi) = \OBox{\beta} (\phi))$.  
	\end{description}
	
	%	Let $\nabla$ be a set of certain \textit{distinguished elements} of $A$.
	%	The class of all $\LPBool_\emptyset$ algebras with singleton $\nabla$ is denoted by $\Algebra_{\LPBool_\emptyset}^N$.
\end{definition}

\begin{definition}
	The interpretation $I$ can be extended to $\tilde{I}: \TermsBool \to T[\VTerms]$ such that for every $t \in \TermsBool$, $\tilde{I}(t)$ is defined as follows:
	\begin{description}
		\item[1.] $\tilde{I}(\justVarOne):= \justVarOne$ for $\justVarOne \in\VTerms$ and $\tilde{I}(\justConsThree):=I(\justConsThree)$ for $\justConsThree \in \CTerms \cup \{ \Abot \}$.
		
		\item[2.] $\tilde{I}(s \ast t):= \tilde{I}(s) \ast_{\T[\VTerms]} \tilde{I}(t)$ where $\ast\in\{+,\cdot\}$.
		
		\item[3.] $\tilde{I}(\ast  s):= \ast_{\T[\VTerms]}\tilde{I}(s)$ where $\ast\in\{-,!\}$.
		
	\end{description}
	
\end{definition}

\begin{definition}
	A valuation is a  function $\Val : \Prop \to A$. The assignment $\BIExt : \Formulae \to A$ on $\Algebra  = (A, \AbotElement, \Anot, \Ajoin, \OBox{\alpha}, I)_{\alpha \in T[\VTerms]}$ is defined as follows:
	\begin{align*}
		\BIExt(p) &= \Val(p),\\
		\BIExt(\bot) &= \AbotElement,\\
		\BIExt(\phi \limplies \psi) &= \BIExt(\phi) \Aimplise \BIExt(\psi), \\
		\BIExt(t: \phi) &= \OBox{\tilde{I}(t)}(\phi), \\
		%	\ValExt(s \lequal t) = \left{ 
		\BIExt(s \lequal t) & = 
		\begin{cases}
			\AtopElement & \text{ if } \tilde{I}(s) \stackrel{p}{=} \tilde{I}(t)\\
			\AbotElement & \text{ if }  \tilde{I}(s) \stackrel{p}{\neq} \tilde{I}(t).
		\end{cases} 
		%	\ValExt(s \lequal t) = \AtopElement &\text{ iff } \ValExt(s: \phi) = \ValExt(t: \phi), \text{for all } \phi.
	\end{align*}
\end{definition}

\begin{definition}
	A polynomial $\LPBool_\CS$ algebra $\Algebra = (A, \AbotElement, \Anot, \Ajoin, \OBox{\alpha}, I)_{\alpha \in T[\VTerms]}$ is a pre polynomial $\LPBool_\CS$ algebra that satisfies the following condition. For all $\phi \in \FormulaeBool$, all $t \in \TermsBool$, all $\alpha \in T[\VTerms]$, and  all valuation $\Val: \Prop \to A$:
	\begin{description}
			\item[Al-j4-$\LPBool_\CS$.] $\OBox{\tilde{I}(t)}(\phi) \leq \OBox{\tinspect_{T[\VTerms]} \tilde{I}(t)} (t:\phi)$,
			
		\item[Al-jT-$\LPBool_\CS$.] $\OBox{\alpha}(\phi) \leq \BIExt(\phi).$
		
	\end{description}
	The class of all polynomial $\LPBool_\CS$ algebras with singleton $\nabla  = \{\AtopElement\}$ is denoted by $\Algebra_{\LPBool_\CS}^{\sf poly}$.
\end{definition}

\begin{definition}[Validity]
	A formula $\phi$ is true in the algebra $\Algebra = (A, \AbotElement, \Anot, \Ajoin, \OBox{\alpha}, I)_{\alpha \in T[\VTerms]}$, denoted by $\Algebra \models \phi$, if $\BIExt(\phi) = \AtopElement$ for every $\theta$. A formula $\phi$ is valid in the class $\Algebra_{\LPBool_\CS}^{\sf poly}$, denoted by $\Algebra_{\LPBool_\CS}^{\sf poly} \models \phi$, if $\Algebra \models \phi$ for every algebra $\Algebra \in \Algebra_{\LPBool_\CS}^{\sf poly}$.
	%	A formula $\phi$ is true in the algebra $\Algebra$, denoted by $\Algebra \models \phi$, if $\OValExt{\phi} \in \nabla$. A formula $\phi$ is valid in the class $\Algebra_{\LP_\CS}^N$, denoted by $\Algebra_{\LP_\CS}^N \models \phi$, if $\Algebra \models \phi$ for every algebra $\Algebra \in \Algebra_{\LP_\CS}^N$.
\end{definition}

%%%%%%%%%%%%%%%%%%%%%%%%%%%%%%%%%%%%%%%%%%%%%%
\subsection{Completeness of $\LPBool_\CS$}

Now we show that $\LPBool_\CS$ is characterized by the class $\Algebra_{\LPBool_\CS}^{\sf poly}$.

\begin{definition}
	Given $\phi \in \FormulaeBool$ and $t \in \TermsBool$, let 
	\[
	\OCl{\phi} := \{ \psi \in \FormulaeBool \mid \vdash_{\LPBool_\CS} \phi \liff\psi \},
	\]
	\[
	\OCl{t} := \{ s \in \TermsBool \mid \vdash_{\LPBool_\CS} s \lequal t \}.
	\]
	First define $\T = (\OCl{\TermsBool}, \ttsum_\T, -_\T, 0_\T,  \tapp_\T,  \tinspect_\T)$ as follows:
	\begin{align*}
	\OCl{\TermsBool} &:= \{ \OCl{t} \mid t \in \TermsBool \}, \\
	\OCl{s} \ttsum_\T \OCl{t} &:= \OCl{s \ttsum t}, \\
	-_\T \OCl{t}  &:= \OCl{- t}, \\
	0_\T  &:= \OCl{\Abot}, \\
	\OCl{s} \tapp_\T \OCl{t} &:= \OCl{s \tapp t}, \\
	\tinspect_\T \OCl{t} &:= \OCl{\tinspect t},
	\end{align*}
	Now consider the structure $\OCl{\TermsBool}[\VTerms]$. The definition of interpretation $I$ is similar to that is given in  Definition \ref{def: Tarski-Lindenbaum algebra for LPBool_emptyset}, for the Tarski-Lindenbaum algebra of $\LPBool_\CS$, as follows:
	\begin{align*}
	I(\justConsThree) &:= \OCl{\justConsThree}, \text{ for } \justConsThree \in \CTerms \cup \{ \Abot \}.  
	\end{align*}
	Note that the function $\tilde{I} : \TermsBool \to \OCl{\TermsBool}[\VTerms]$ is surjective.
	
	The Tarski-Lindenbaum algebra for $\LPBool_\CS$ is defined as follows:	
	
	\[
	\Algebra_{\LPBool_\CS}^p := ( \OCl{\FormulaeBool}, \AbotElement, \Anot, \Ajoin, \OBox{\alpha}, I)_{\alpha \in \OCl{\TermsBool}[\VTerms]}
	\]
	where $\Ajoin, \Anot$, and $\AbotElement$ are defined similar to Definition \ref{def: Tarski-Lindenbaum algebra for HLP}, and
	\[
	\OCl{\FormulaeBool} := \{ \OCl{\phi} \mid \phi \in \FormulaeBool \},
	\]
	\[
	\OBox{\alpha} (\phi) :=   \OCl{t:\phi}, \quad \text{ where } \alpha = \tilde{I}(t) \text{ for some } t \in \TermsBool.
	\]

	Let $\nabla =  \{ \OCl{\top} \}$. %$\nabla =  \{ \OCl{\phi} \mid \vdash \phi \}$.
\end{definition}

\begin{lemma}\label{lem: tilde(I)(t) is t^v-I}
	Let $\Algebra_{\LPBool_\CS}^p := ( \OCl{\FormulaeBool}, \AbotElement, \Anot, \Ajoin, \OBox{\alpha}, I)_{\alpha \in \OCl{\TermsBool}[\VTerms]}$ be the Tarski-Lindenbaum algebra for $\LPBool_\CS$ and let $v : \VTerms \to \OCl{\TermsBool}$ be a function. Then, for every $t \in \TermsBool$:
	$$v(\tilde{I}(t)) = t_I^v.$$ 
	Where  $t_I^v$ is defined as in Definition \ref{def: t^I_v Section 5}.
\end{lemma}
\begin{proof}
	The proof is by induction on the complexity of $t$. The base cases $t \in \VTerms$ and $t \in \CTerms \cup \{\Abot\}$ are trivial. For the induction step we only show the case that $t = !s$.
	\[
	v(\tilde{I}(!s)) = v(!_{\OCl{\TermsBool}[\VTerms]} \tilde{I}(s)) = !_{\OCl{\TermsBool}[\VTerms]} v(\tilde{I}(s))) = !_{\OCl{\TermsBool}} v(\tilde{I}(s))) = !_{\OCl{\TermsBool}} s_I^v = (!s)_I^v.
	\]
	The third equation follows from the fact that $!_{\OCl{\TermsBool}[\VTerms]}$ is an extension of $!_{\OCl{\TermsBool}}$ and $v(\tilde{I}(s))  \in \OCl{\TermsBool}$.
	\qed
 \end{proof}

%\begin{lemma}\label{lem: tilde(I) is a correspondence}
%	Let $\Algebra_{\LPBool_\CS} := ( \OCl{\FormulaeBool}, \AbotElement, \Anot, \Ajoin, \OBox{\alpha}, I)_{\alpha \in \OCl{\TermsBool}[\VTerms]}$ be the Tarski-Lindenbaum algebra for $\LPBool_\CS$. Then, for every $s,t \in \TermsBool$, if $\tilde{I}(s) \stackrel{p}{=} \tilde{I}(t)$ implies $\OCl{s} = \OCl{t}$. 
%	%
%%	\begin{center}
%%		$\tilde{I}(s) = \tilde{I}(t)$ if{f} $\OCl{s} = \OCl{t}$
%%	\end{center}
%\end{lemma}
%\begin{proof}
%	
%	
%	 \qed
%\end{proof}

\begin{corollary}
	The operators $\OBox{\alpha}$, for each $\alpha \in \OCl{\TermsBool}[\VTerms]$, in $\Algebra_{\LPBool_\CS}^p$ are well-defined.
\end{corollary}
\begin{proof}
	We first show that $\OBox{\alpha}$ is well-defined on $\OCl{\TermsBool}[\VTerms]$. Suppose that $s, t \in \tilde{I}^{-1} (\alpha)$, for some $s,t \in \TermsBool.$ Then, $\tilde{I}(s) \stackrel{p}{=} \tilde{I}(t)$. Hence, $v(\tilde{I}(s)) = v(\tilde{I}(t))$, for every $v : \VTerms \to \OCl{\TermsBool}$. By Lemma \ref{lem: tilde(I)(t) is t^v-I}, we get $s_I^v = t_I^v$, for every $v : \VTerms \to \OCl{\TermsBool}$. Now define $v  : \VTerms \to \OCl{\TermsBool}$ as follows:
	$$v(x) = \OCl{x}.$$
	For this function, we have $s_I^{v} = t_I^{v}$. By \eqref{eq: t^{v}_{I} = \OCl{t} completeness of LPBool_emptyset} in the proof of Theorem \ref{thm: completeness LPBool over arbitrary Boolean algebra}, we obtain $\OCl{s} = \OCl{t}$, and hence $\vdash_{\LPBool_\CS} t \approx s$. Thus, $\vdash_{\LPBool_\CS} t:\phi \liff s:\phi$, for every $\phi \in \FormulaeBool$. Therefore $\OCl{t:\phi} = \OCl{s:\phi}$, and then $\OBox{\tilde{I}(t)} (\phi) = \OBox{\tilde{I}(s)} (\phi)$.
	
	Let $\alpha = \tilde{I}(t)$, for some $t \in \TermsBool$. We now show that $\OBox{\alpha}$ is well-defined on $\FormulaeBool$. Suppose that $\phi = \psi$. Then, $\vdash_{\LPBool_\CS} t:\phi \liff t:\psi$. Thus, $\OCl{t:\phi} = \OCl{t:\psi}$. Hence, $\OBox{\alpha} (\phi) = \OBox{\alpha} (\psi)$.
	  \qed
\end{proof}

%\begin{lemma}
%	Let $\Algebra_{\LPBool_\CS} := ( \OCl{\FormulaeBool}, \AbotElement, \Anot, \Ajoin, \OBox{\alpha}, I, \Val)_{\alpha \in \OCl{\TermsBool}[\VTerms]}$ be the Tarski-Lindenbaum algebra for $\LPBool_\CS$. Then, for every $t \in \TermsBool \setminus \VTerms$
%	%
%	\[
%	\tilde{I}(t) = \OCl{t}
%	\]
%\end{lemma}

\begin{lemma}
	The Tarski-Lindenbaum algebra $\Algebra_{\LPBool_\CS}^p$ is a polynomial $\LPBool_\CS$ algebra.
\end{lemma}
\begin{proof}
	We only check $Al$-$Appl$-$\LPBool_\CS$ and $Al$-$\Atop$-$\LPBool_\CS$. The proof of other cases is simple.
	\begin{itemize}
		\item Condition $Al$-$Appl$-$\LPBool_\CS$.
		\begin{align*}
		& \OBox{\tilde{I}(s)}(\OCl{\phi} \Aimplise \OCl{\psi}) \Ameet \OBox{\tilde{I}(t)}(\OCl{\phi}) \\ &=
		\OBox{\tilde{I}(s)}(\OCl{\phi \limplies \psi}) \Ameet \OBox{\tilde{I}(t)}(\OCl{\phi}) \\ &=
		\OCl{\jbox{s} (\phi \limplies \psi)}  \Ameet  \OCl{\jbox{t} \phi}\\  &=
		\OCl{\jbox{s} (\phi \limplies \psi) \wedge \jbox{t} \phi} \\  &\leq
		\OCl{\jbox{s \tapp t} \psi} \\ &=  
		\OBox{\tilde{I}(s \tapp t)} \OCl{\psi} 
		\end{align*}
		
		\item Condition $Al$-$\Atop$-$\LPBool_\CS$. Suppose that $\vdash_{\LPBool_\CS} \phi$. Thus, by the rule Int, $\vdash_{\LPBool_\CS} \Atop:\phi$. Then
		\[
		\OBox{1_\T}(\phi) = \OBox{\OCl{\Atop}} (\phi) =  \OCl{\Atop:\phi} =  \OCl{\top}. 
		\]
		\qed
	\end{itemize}
	
\end{proof}

\begin{theorem}[Soundness and completeness]\label{thm: completeness LPBool-CS over arbitrary Boolean algebra}
	$\vdash_{\LPBool_\CS} \phi$ iff $\Algebra_{\LPBool_\CS}^{\sf poly} \models \phi$.
\end{theorem}
\begin{proof}
	Soundness is straightforward. For completeness, define $\Val$ as follows:
	\begin{align*}
		\OVal{p} & := \OCl{p}, \qquad p \in \Prop.
	\end{align*}
	Then we show the Truth Lemma. For every $\phi \in \FormulaeBool$:
		\[
		\BIExt(\phi) = \OCl{\phi}.
		\]
	The proof is by induction on the complexity of $\phi$. We only check two cases:
	\begin{itemize}
		\setlength\itemsep{0.1cm}
		\item Case $\phi = t:\psi$: 
		
		$\BIExt(t: \phi) = \OBox{ \tilde{I}(t) }(\phi) =  \OCl{t:\phi}$.
		
		\item Case $\phi = s \lequal t$:
		
		Suppose that $\BIExt(s \lequal t) = \OCl{\top}$. Then $\tilde{I}(s) \stackrel{p}{=} \tilde{I}(t)$, and hence $v(\tilde{I}(s)) = v(\tilde{I}(t))$, for every $v : \VTerms \to \OCl{\TermsBool}$. Thus, by Lemma \ref{lem: tilde(I)(t) is t^v-I}, we get  $s^v_I = t^v_I$, for every $v : \VTerms \to \OCl{\TermsBool}$. Then, for $v(\justVarOne) = \OCl{x}$ we obtain $s^v_I = t^v_I$. By \eqref{eq: t^{v}_{I} = \OCl{t} completeness of LPBool_emptyset} in the proof of Theorem \ref{thm: completeness LPBool over arbitrary Boolean algebra}, we get $\OCl{s} = \OCl{t}$, and hence $\OCl{ s \lequal t} = \OCl{\top}$.\footnote{Note that the definition of the interpretation $I$ in the Tarski-Lindenbaum algebras $\Algebra_{\LPBool_\CS}^f$ (cf. Definition \ref{def: Tarski-Lindenbaum algebra for LPBool_emptyset}) and $\Algebra_{\LPBool_\CS}^p$ are the same.}
		
		For the converse, from $\OCl{ s \lequal t} = \OCl{\top}$ it follows that  $\vdash_{\LPBool_\CS} s \lequal t$. By soundness theorem \ref{thm: completeness LPBool over arbitrary Boolean algebra}, we have $s^v_I = t^v_I$ for every $\LPBool_\CS$ algebra $\Algebra$ over $\T$, every valuation $\Val$, and every $v : \VTerms \to T$. Thus, for the Tarski-Lindenbaum algebra  $\Algebra_{\LPBool_\emptyset}$ over $\OCl{\TermsBool}$ (cf. Definition \ref{def: Tarski-Lindenbaum algebra for LPBool_emptyset}), we have $s^v_I = t^v_I$, for every valuation $\Val$ and every $v : \VTerms \to \OCl{\TermsBool}$. Again since the definition of interpretation $I$ in the Tarski-Lindenbaum algebras $\Algebra_{\LPBool_\CS}$ and $\Algebra_{\LPBool_\CS}^p$ are the same, we get $s^v_I = t^v_I$, for every $v : \VTerms \to \OCl{\TermsBool}$. Hence, by Lemma \ref{lem: tilde(I)(t) is t^v-I}, $\tilde{I}(s) \stackrel{p}{=} \tilde{I}(t)$. Thus, $\BIExt(s \lequal t) = \OCl{\top}$.\\
		From the above two arguments, it follows that $\BIExt(s \lequal t) = \OCl{s \lequal t}$.
	\end{itemize}
		 \qed
		
\end{proof}

%%%%%%%%%%%%%%%%%%%%%%%%%%%%%%%%%%%%%%%%%%%%%%
\subsection{Bi-representation theorem for full $\LPBool_\CS$ algebras}

At the end, we present a bi-representation theorem for polynomial $\LPBool_\CS$ algebras similar to Section \ref{sec: Bi-representation theorem LPB-emptyset}.

Given $\T = (T, 0_\T, -_\T, \ttsum_\T, \tapp_\T, \tinspect_\T)$ and $\T' = (T', 0_{\T'}, -_{\T'}, \ttsum_{\T'}, \tapp_{\T'}, \tinspect_{\T'})$, note that every isomorphism $g: \T \to \T'$ can be naturally extended to an isomorphism between the polynomial structures $\T[\VTerms]$ and $\T'[\VTerms]$ where
\[
g(\justVarOne) := \justVarOne, \qquad \mbox{ for every } \justVarOne \in \VTerms.
\]
By the abuse of notation we denote this extension again by $g$.

\begin{definition}
	Let $\Algebra = (A, \AbotElement, \Anot, \Ajoin, \OBox{\alpha}, I)_{\alpha \in T[\VTerms]}$ and $\Algebra' = (A', \AbotElement', \Anot', \Ajoin', \Box'_{\alpha}, I')_{\alpha \in T'[\VTerms]}$ be polynomial $\LPBool_\CS$ algebras over $T[\VTerms]$ and $T'[\VTerms]$ respectively, where $\T = (T, 0_\T, -_\T, \ttsum_\T, \tapp_\T, \tinspect_\T)$ and $\T' = (T', 0_{\T'}, -_{\T'}, \ttsum_{\T'}, \tapp_{\T'}, \tinspect_{\T'})$. A bi-isomorphism  between $\Algebra$ and $\Algebra'$ is a pair of Boolean isomorphisms $(f,g)$ such that $f: (A, \AbotElement, \Anot, \Ajoin) \to (A', \AbotElement', \Anot', \Ajoin')$ and  $g: (T, 0_\T, -_\T, \ttsum_\T) \to (T', 0_{\T'}, -_{\T'}, \ttsum_{\T'})$ satisfying the following conditions:
	\begin{align*}
		f(\OBox{\alpha}(\phi)) &= \Box'_{g(\alpha)}(\phi), \qquad \mbox{ for $\alpha \in T[\VTerms]$ and $\phi \in \FormulaeBool$, }
		\\
		g(I(\justConsThree)) &= I'(\justConsThree), \qquad\quad \mbox{ for $\justConsThree \in \CTerms \cup \{ \Abot\}$, }
		\\
		g(\alpha \tapp_{\T[\VTerms]} \beta) &= g(\alpha) \tapp_{\T'[\VTerms]} g(\beta) ,  \qquad\mbox{ for $\alpha, \beta \in T[\VTerms]$, }
		\\
		g(\tinspect_{\T[\VTerms]} \alpha) &=	\tinspect_{\T'[\VTerms]} g(\alpha), \qquad\mbox{ for $\alpha \in T[\VTerms]$}.
	\end{align*}
	
\end{definition}

\begin{definition}
	Given two set algebras $(A, \emptyset, \setminus, \cup)$ and $(B, \emptyset, \setminus, \cup)$, a polynomial $\LPBool_\CS$ set algebra is a structure 
	$$\Algebra = (A, \emptyset, \setminus, \cup, \OBox{\alpha}, I)_{\alpha \in B[\VTerms]}$$ 
	over $B[\VTerms]$, where $\mathcal{B} = (B, \emptyset, \setminus, \cup, \tapp_\mathcal{B}, \tinspect_\mathcal{B})$, satisfying the conditions of Definition \ref{def: LPB-CS algebra}.
\end{definition}

Similar to Section \ref{sec: Bi-representation theorem LPB-emptyset}, we construct two isomorphisms for a given polynomial $\LPBool_\CS$ algebra $\Algebra$ over $T[\VTerms]$: One for the algebra $\Algebra$ and one for the polynomial structure $\T[\VTerms]$. This leads to a bi-isomorphism of polynomial $\LPBool_\CS$ algebras.

\begin{theorem}[Bi-Representation Theorem]\label{thm: Representation Theorem LPBool-CS}
	Every polynomial $\LPBool_\CS$ algebra is bi-isomorphic to a polynomial $\LPBool_\CS$ set algebra.
	%For every $\LPBool_\emptyset$ algebra $\Algebra = (A, \AbotElement, \Anot, \Ajoin, \OBox{\alpha}, I)_{\alpha \in T}$ over $\T = (T, 0_\T, -_\T, \ttsum_\T, \tapp_\T, \tinspect_\T)$, there exists a power set $\LPBool_\emptyset$ algebra $\mathcal{B} = (B,  \Ajoin', \Anot', \AbotElement', \OBox{\alpha}, I')_{\alpha \in \mathcal{P}_\T}$ which is bi-isomorphic to $\Algebra$.
\end{theorem}

\begin{proof}
	%	Follows from the Stone representation theorem and Lemma \ref{lem: Bi-Representation Bool}. 
	Suppose that $\Algebra = (A, \AbotElement, \Anot, \Ajoin, \OBox{\alpha}, I)_{\alpha \in T[\VTerms]}$ is a polynomial $\LPBool_\CS$ algebra over $T[\VTerms]$, where $\T = (T, 0_\T, -_\T, \ttsum_\T, \tapp_\T, \tinspect_\T)$. Considering the Boolean algebra $(T, 0_\T, -_\T, \ttsum_\T)$, by the Stone representation theorem, there exists an isomorphism $g$ from $(T, 0_\T, -_\T, \ttsum_\T)$ to a set algebra $\mathcal{P} = (P, \emptyset, \setminus, \cup)$.  Similar to the proof of Theorem \ref{thm: Representation Theorem LPBool}, $g$ is extended to an isomorphism from $\T = (T, 0_\T, -_\T, \ttsum_\T, \tapp_\T, \tinspect_\T)$ to  $(P, \emptyset, \setminus, \cup, \tapp_\mathcal{P}, \tinspect_\mathcal{P})$. Let $\mathcal{P}_{\T} = (P, \emptyset, \setminus, \cup, \tapp_\mathcal{P}, \tinspect_\mathcal{P})$. As mentioned before, $g$ can be further extended to an isomorphism from $\T[\VTerms]$ to $\mathcal{P}_\T[\VTerms]$.

	Next, by the Stone representation theorem, $(A, \AbotElement, \Anot, \Ajoin)$ is isomorphic to a set algebra $(B, \emptyset, \setminus, \cup)$ via isomorphism $f$. Define $\OBox{\alpha'}^\prime$ (for $\alpha' \in P[\VTerms]$) and the interpretation $I'$ as follows: 
	\begin{equation}\label{eq: induced relation via bi-isomorphism Bool-CS}
		\OBox{g(\alpha)}^\prime (\phi) := f( \OBox{\alpha}(\phi)), \qquad\mbox{ for $\alpha \in T[\VTerms]$. }
	\end{equation}
	\begin{equation}\label{eq: induced relation via bi-isomorphism Bool-interpretation-CS}
		I'(\justConsThree) := g(I(\justConsThree)), \qquad\mbox{ for $\justConsThree \in \CTerms \cup \{ \Abot\}$. }
	\end{equation}
	Let $\mathcal{B} = (B, \emptyset, \setminus, \cup, \OBox{\alpha'}^\prime, I')_{\alpha' \in P[\VTerms]}$. It is not difficult to show that for any $t \in \TermsBool$:
	\begin{equation}\label{eq: fact bi-isomorphism Bool-interpretation-CS}
		\tilde{I'} (t) = g( \tilde{I}(t) ).
	\end{equation}
	The proof is by a simple induction on the complexity of $t$.
	Then, we show that $\mathcal{B}$ is a polynomial $\LPBool_\CS$  algebra. %The proof is similar to that of Lemma \ref{lem:pre-stone LP-empty}.
	To this end we only check conditions $Al$-$Appl$-$\LPBool_\CS$, $Al$-$\CS$-$\LPBool_\CS$, $Al$-$j4$-$\LPBool_\CS$ and $Al$-$jT$-$\LPBool_\CS$ from Definition \ref{def: LPB-CS algebra} (other cases can be verified similarly).
	
	%	For \textbf{EqTm}, suppose $g(s) = g(t)$. Then $s = t$, and
	%	%
	%	\[
	%	\OBox{g(s)}^\mathcal{B} (f(\AElementOne )) = \{ f(\AElementTwo ) \mid \AElementTwo  \in \OBox{s}^\Algebra(\AElementOne ) \} = \{ f(\AElementTwo ) \mid \AElementTwo  \in \OBox{t}^\Algebra(\AElementOne ) \} = \OBox{g(t)}^\mathcal{B} (f(\AElementOne )).
	%	\]
	
	For $Al$-$Appl$-$\LPBool_\CS$, suppose that $\AElementOne ,\AElementTwo  \in \mathcal{A}$:
	\begin{eqnarray*}
		& \OBox{g(\alpha)}^\prime(\phi \limplies \psi) \Ameet \OBox{g(\beta)}^\prime(\phi)=\\		
		& f(\OBox{\alpha}(\phi \limplies \psi)) \Ameet f(\OBox{\beta}(\phi)) =
		\\		
		& f(\OBox{\alpha}(\phi \limplies \psi) \Ameet \OBox{\beta}(\phi)) \leq \\ 
		& f(\OBox{\alpha \tapp_{\T[\VTerms]} \beta}(\psi))= \OBox{g(\alpha) \tapp_{\mathcal{P}_\T[\VTerms]} g(\beta)}^\prime(\psi).
	\end{eqnarray*}
	
	For $Al$-$\CS$-$\LPBool_\CS$, suppose $\justConsThree:\phi \in \CS$, then we have
	\[
	\OBox{I'(\justConsThree)}^\prime (\phi) = \OBox{g(I(\justConsThree))}^\prime (\phi) = f(\OBox{I(\justConsThree)}(\phi)) = f(\AtopElement) = B.
	\]
	
	For $Al$-$j4$-$\LPBool_\CS$, note that since $g$ is an isomorphism from $\T[\VTerms]$ to $\mathcal{P}_\T[\VTerms]$, for every $t \in \TermsBool$ there is $\alpha \in \T[\VTerms]$ such that $g(\alpha) = \tilde{I'}(t)$. Furthermore, from \eqref{eq: fact bi-isomorphism Bool-interpretation-CS} it follows that $\tilde{I} (t) = g^{-1} (\tilde{I'}(t)) = g^{-1} (g(\alpha)) = \alpha$. Thus,
	\begin{align*}
		\OBox{\tilde{I'}(t)}^\prime (\phi) & = \OBox{g(\alpha)}^\prime (\phi) = f( \OBox{\alpha}(\phi)) = f( \OBox{\tilde{I}(t)}(\phi)) \leq f(\OBox{\tinspect_{T[\VTerms]} \tilde{I}(t)} (t:\phi)) = f(\OBox{\tinspect_{T[\VTerms]} \alpha} (t:\phi)) \\ & = \OBox{g(\tinspect_{T[\VTerms]} \alpha)}^\prime (t:\phi)  = \OBox{\tinspect_{\mathcal{P}_\T[\VTerms]} g(\alpha)}^\prime (t:\phi).
	\end{align*}
%	\[
%	\OBox{\tilde{I'}(t)}^\prime (\phi) = \OBox{g(\alpha)}^\prime (\phi) = f( \OBox{\alpha}(\phi)) = f( \OBox{\tilde{I}(t)}(\phi)) \leq f(\OBox{\tinspect_{T[\VTerms]} \tilde{I}(t)} (t:\phi)) = f(\OBox{\tinspect_{T[\VTerms]} \alpha} (t:\phi)) = \OBox{g(\tinspect_{T[\VTerms]} \alpha)}^\prime (t:\phi)
%	\]

	For $Al$-$jT$-$\LPBool_\CS$, suppose that $\Val' : \Prop \to B$ is  a valuation on $\mathcal{B}$. Then, $\Val = f^{-1} \circ \Val'$ would be a valuation on $\Algebra$. Thus, we have
	\[
	\OBox{g(\alpha)}^\prime (\phi) = f( \OBox{\alpha}(\phi)) \leq f(\tilde{\theta} (\phi)) = \tilde{\theta'} (\phi).
	\]
	The above inequality follows from the fact that $f$ is an isomorphism  and $\OBox{\alpha}(\phi) \leq \tilde{\theta} (\phi)$.
	
	It is obvious that $\mathcal{B}$ is a polynomial $\LPBool_\CS$ set algebra.	\qed
\end{proof}

\begin{theorem}
	$\LPBool_\CS$ is complete with respect to polynomial $\LPBool_\CS$ set algebras.
\end{theorem}
\begin{proof}
	Follows from the Bi-representation Theorem \ref{thm: Representation Theorem LPBool-CS} and the Completeness Theorem \ref{thm: completeness LPBool-CS over arbitrary Boolean algebra}. \qed
\end{proof}

%%%%%%%%%%%%%%%%%%%%%%%%%%%%%%%%%%%%%%%%%%%%%%%%%%%%%%%%%%%%%%%%%%%%%%%%%%%%%%%%%%
\section*{Conclusion}

In this paper we have introduced an algebraic semantics for the logic of proofs. These algebraic models are obtained from Boolean algebras by adding countably infinite functions on formulas. We have also extended the language of $\LP$ in a way that proof terms constitute a Boolean algebra. The language of the resulting logic $\LPBool$ includes an equality predicate. We have proved completeness theorems and certain generalizations of Stone's representation theorem for all of the aforementioned algebras.

Comparing modal algebras with $\LP$ algebras, we see that the class of $\LP$ algebras does not form a variety; in other words, the $\LP$ algebras are not equationally definable. While pre $\LP_\CS$ algebras form a variety in the above sense, the condition $Al$-$jT$-$\LP_\CS$ in full algebras cannot be expressed by an equation.

There remain many research problems deserving further study. For instance, a natural question to ask is whether it is possible to present regular algebras with a finite set $\nabla$ for $\LP_\CS$. Another problem is to present algebraic models for other justification logics. Although it is straightforward to present algebraic semantics for fragments of $\LP$, such as $\JLCS{JT}{}$ and $\JLCS{J4}{}$,\footnote{To this end, it is enough to omit the condition $Al$-$j4$ to obtain $\JLCS{JT}{}$-algebras, and omit the condition $Al$-$jT$ to obtain $\JLCS{J4}{}$-algebras.} it would be interesting to extend the results of this paper to other justification logics, such as extensions of $\LP$.

Another interesting direction to extend this work is to consider fuzzy justification logics. Various fuzzy justification logics have been introduced in the literature (see \cite{Gha14fuzzy-arXiv,Ghari-IGPL-2016,Pischke-SL-2020,Pischke-IGPL-2021}), however none of these logics have algebraic semantics. One might expect that $\LP$ algebras can be modified in order to construct fuzzy algebraic semantics, such as  MV algebras, for a fuzzy version of $\LP$.

\section*{Acknowledgement}

We wish to express our deep gratitude to Nicholas Pischke for reading an earlier draft of the paper and providing many precious suggestions.

\bibliography{library}

\end{document}